\documentclass[hidelinks,onefignum,onetabnum,
  ]{siamart250211}

\headers{Inexact Riemannian DCA}{B. Jiang, M. Xu, X. Cai, and Y.-F. Liu}

\title{
An Inexact Proximal Framework for Nonsmooth Riemannian Difference-of-Convex Optimization
}

\author{Bo Jiang\thanks{Ministry of Education Key Laboratory of NSLSCS, School of Mathematical Sciences, Nanjing Normal University, Nanjing 210023, China 
  (\email{jiangbo@njnu.eud.cn}, \email{caixingju@njnu.edu.cn}).}
\and  Meng Xu\thanks{State Key Laboratory of Scientific and Engineering Computing, Institute of Computational Mathematics and Scientific/Engineering Computing, Academy of Mathematics and Systems Science, Chinese Academy of Sciences, Beijing 100190, China 
  (\email{xumeng22@mails.ucas.ac.cn}).}
  \and Xingju Cai\footnotemark[1] 
\and  Ya-Feng Liu\thanks{Ministry of Education Key Laboratory of Mathematics and Information Networks, School of
Mathematical Sciences, Beijing University of Posts and Telecommunications,
Beijing 102206, China (\email{yafengliu@bupt.edu.cn}).}}

\ifpdf
\hypersetup{
  pdftitle={Inexact Riemannian DCA},
  pdfauthor={B. Jiang, M. Xu, X. Cai, and Y.-F. Liu}
}
\fi

\usepackage{lipsum}
\usepackage{amsfonts}
\usepackage{graphicx}
\usepackage{epstopdf}

\usepackage[algo2e, linesnumbered, ruled, vlined]{algorithm2e}

\ifpdf
  \DeclareGraphicsExtensions{.eps,.pdf,.png,.jpg}
\else
  \DeclareGraphicsExtensions{.eps}
\fi

\newsiamremark{remark}{Remark}
\newsiamremark{assumption}{Assumption}
\newsiamremark{example}{Example}
\newsiamremark{hypothesis}{Hypothesis}
\crefname{hypothesis}{Hypothesis}{Hypotheses}
\newsiamthm{claim}{Claim}
\newsiamremark{fact}{Fact}
\crefname{fact}{Fact}{Facts}

\usepackage{graphicx}
\usepackage[caption=false]{subfig} 
\usepackage{pifont}
\usepackage{array} 
\usepackage{booktabs}
\usepackage{arydshln}
\newsavebox\CBox

\usepackage{enumitem}
\DeclareMathOperator*{\argmin}{arg\,min}

\DeclareMathOperator{\grad}{grad}
\DeclareMathOperator{\prox}{prox}

\newcommand{\iprod}[2]{\left\langle #1, #2 \right \rangle}
\newcommand{\iprods}[2]{\langle #1, #2  \rangle}

\newcommand{\revgpt}[1]{{\color{black}{#1}}}

\newcommand{\tran}{\top}

\newcommand{\ba}{\begin{array}}
\newcommand{\ea}{\end{array}}

\newcommand{\bit}{\begin{itemize}}
\newcommand{\eit}{\end{itemize}}
\newcommand{\be}{\begin{equation}}
\newcommand{\ee}{\end{equation}}
\newcommand{\bea}{\begin{eqnarray}}
\newcommand{\eea}{\end{eqnarray}}

\newcommand{\st}{\mathrm{s.t.}}
\newcommand{\mtr}{\mathrm{tr}}

\newcommand{\proj}{\mathsf{Proj}}

\newcommand{\stief}{\mathcal{S}^{n,r}}

\newcommand{\rgrad}{\mathrm{grad}\,}

\newcommand{\Rmn}[1]{\uppercase\expandafter{\romannumeral#1}}

\renewcommand{\theenumi}{\roman{enumi}}

\newcommand{\Ccal}{\mathcal{C}}

\newcommand{\Ecal}{\mathcal{E}}

\newcommand{\Mcal}{\mathcal{M}}

\newcommand{\Rcal}{\mathcal{R}}
\newcommand{\Retr}{\mathrm{Retr}}
\newcommand{\Scal}{\mathcal{S}}

\newcommand{\Rbb}{\mathbb{R}}

\newcommand{\Ftt}{\mathtt{F}}

\newcommand{\dist}{\mathrm{dist}}
\newcommand{\tangent}{\mathrm{T}}
\newcommand{\capp}{\upsilon}

\newcommand{\normmm}[1]{{\left\vert\kern-0.25ex\left\vert\kern-0.25ex\left\vert #1 
   \right\vert\kern-0.25ex\right\vert\kern-0.25ex\right\vert}}
\begin{document}
\maketitle
\begin{abstract}
Nonsmooth Riemannian optimization has attracted increasing attention, especially in problems with sparse structures. While existing formulations typically involve convex nonsmooth terms, incorporating nonsmooth difference-of-convex (DC) penalties can enhance recovery accuracy. In this paper, we study a class of nonsmooth Riemannian optimization problems whose objective is the sum of a smooth function and a nonsmooth DC term. We establish, for the first time in the manifold setting, the equivalence between such DC formulations (with suitably chosen nonsmooth DC terms) and their $\ell_0$-regularized or $\ell_0$-constrained counterparts. To solve these problems, we propose an inexact Riemannian proximal DC (iRPDC) algorithmic framework, which returns an $\epsilon$-Riemannian critical point within $\mathcal{O}(\epsilon^{-2})$ outer iterations.  Within this framework, we develop several practical algorithms based on different subproblem solvers.  Among them, one achieves an overall iteration complexity of $\mathcal{O}(\epsilon^{-3})$, which matches the best-known bound in the literature.  In contrast, existing algorithms either lack provable overall complexity or require $\mathcal{O}(\epsilon^{-3})$ iterations in both outer and overall complexity.   A notable feature of the iRPDC algorithmic framework  is a novel inexactness criterion that not only enables efficient subproblem solutions via first-order methods but also facilitates a linesearch procedure that adaptively captures the local curvature. Numerical results on sparse principal component analysis demonstrate  the modeling flexibility of the DC formulaton and the competitive performance of  the proposed algorithmic framework. 
\end{abstract}

\begin{keywords}
 difference-of-convex optimization, inexact framework, nonsmooth Riemannian optimization, sparse optimization, overall complexity
\end{keywords}

\begin{MSCcodes}
68Q25, 68R10, 68U05
\end{MSCcodes}

\section{Introduction}
In this paper, we study a class of nonsmooth Riemannian difference-of-convex (DC) optimization problems of the form
 \be \label{equ:prob:dc}
\min_{x \in \Mcal}\, \left\{F(x):= f(x) + h(x) - g(x)\right\}, 
\ee
where  $\Mcal$ is a Riemannian submanifold of a finite-dimensional Euclidean space $\Ecal$, which is equipped with the standard inner product $\iprod{\cdot}{\cdot}$ and the induced $\ell_2$-norm $\|\cdot\|$.  Problem \eqref{equ:prob:dc}, including its special case where \(g(\cdot) = 0\), captures a wide range of applications, particularly in signal processing and machine learning; see \cite{d2008optimal,chen2020proximal,chen2020alternating,xiao2021exact,huang2025riemannian} and references therein for more details. 
Throughout this paper, we assume that the functions $f$, $h$, and $g$ satisfy the following assumptions.  
\begin{assumption} \label{assumption:f:h:g}
\begin{enumerate}[leftmargin=*, itemindent=2em, labelsep=0.5em] 
\item [(i)] The  function $f: \Ecal \to \Rbb$ is  smooth, Lipschitz continuous with parameter $L_f^0 \geq 0$, and satisfies the descent inequality  with parameter $L_f\geq 0$, i.e.,  
\be \label{equ:f:descent:euclidean}
f(y) \leq f(x) + \langle \nabla f(x), y - x \rangle + \frac{L_f}{2} \|y - x \|^2, \quad \forall\, x,y \in \Ecal.
\ee
\item [(ii)] The functions \(h, g: \Ecal \to \Rbb\) are convex, possibly nonsmooth, and Lipschitz continuous with parameters ${L_h^0} \geq 0$ and ${L_g^0} \geq 0$, respectively.
The proximal mapping of $h$ and a subgradient of $g$ can be computed efficiently. 
\item [(iii)] The level set $\{x \in \Mcal \mid F(x) \leq \bar F\}$ is compact  for some  $\bar F \in \Rbb$. 
\end{enumerate}
\end{assumption}

\subsection{Motivating examples} \label{section:motivating:examples}
We present two motivating examples of problem \eqref{equ:prob:dc}, whose connections to their sparse counterparts will be discussed in Section \ref{section:DC:sparse}. 
\begin{example}\label{equ:prob:capped:l1}
A typical sparse optimization problem over a manifold involving the $\ell_0$-norm in the objective takes the form \cite{d2008optimal,journee2010generalized}:   
\be \label{prob:sphere:ksparse:l0:regularized}
\min_{x \in \Mcal}\, f(x) + \sigma \|x\|_0,
\ee  
where $\sigma > 0$ is a given parameter and, $\|x\|_0$, the so-called $\ell_0$-norm of $x$, denotes the number of nonzero elements in $x$.  In the Euclidean setting,  this nonconvex term is often approximated by the capped-$\ell_1$ penalty  \cite{peleg2008bilinear},  which is one of the tightest continuous DC relaxations of $\|x\|_0$; see \cite{le2015dc}  for details.  Extending this idea to the   Riemannian case, we consider the following capped-$\ell_1$ penalized model: 
\be \label{prob:sphere:l0:capped:l1}
\min_{x \in \Mcal}\, f(x) + \sigma  \Phi_{\capp}(x), 
\ee 
where $\Phi_{\capp}(x) = \sum_{i} \min\{\upsilon |x_{(i)}|, 1\}$ with a given $\upsilon > 0$, and $x_{(i)}$ is the $i$-th element of $x$. 
Problem \eqref{prob:sphere:l0:capped:l1} is an instance of problem \eqref{equ:prob:dc} with $h(x) =  \sigma \upsilon \|x\|_1$ and 
 $g(x) =  \sigma \sum_i \max\{\upsilon |x_{(i)}|-1, 0\}$, where $\|x\|_1$ is the $\ell_1$-norm of $x$. 
\end{example}
\begin{example}\label{lemma:exact:penatly:sphere}
In many applications, such as sparse principal component analysis (SPCA) \cite{d2004direct}, sparse Fisher's discriminant analysis \cite{cai2021note}, and clustering problems \cite{huang2025riemannian}, strict sparsity constraints are required.  This leads to the following formulation: 
\be  \label{prob:sphere:ksparse} 
\min_{x \in \Mcal}\, f(x) \quad \st \quad \|x\|_0 \leq k,
\ee 
where $k$ is a given positive integer.  Define the largest $k$-norm of $x$ as $\normmm{x}_k:=  |x_{[1]}| +  |x_{[2]}| + \cdots +  |x_{[k]}|$, where $|x_{[i]}|$ is the $i$-th largest element among $\{|x_{(i)}|\}$. 
Observing that $\|x\|_0 \leq k$ is equivalent to the DC constraint $\|x\|_1 - \normmm{x}_k = 0$, the work \cite{gotoh2018dc} reformulated problem \eqref{prob:sphere:ksparse}  by penalizing this constraint in the objective (in the Euclidean setting). Following this idea, we extend it to the manifold setting:
\be \label{prob:sphere:knorm} 
\min_{x \in \Mcal}\, f(x) + \gamma \left( \|x\|_1 -  \normmm{x}_k \right),
\ee
where $\gamma > 0$ is a sparsity penalty parameter. This problem again matches  problem \eqref{equ:prob:dc} with $h(x) = \gamma \|x\|_1$ and $g(x) = \gamma \normmm{x}_k$.  
\end{example}

\subsection{Related works} We briefly review existing works on DC programming and nonsmooth Riemannian optimization.
 \paragraph{DC programming in Euclidean settings} DC programming has been extensively studied since the 1980s; see \cite{le2018dc} and references therein.  A standard DC program corresponds to the formulation \eqref{equ:prob:dc} with $\Mcal$ taken as a closed convex set of $\Ecal$.  The classic approach is the so-called DC algorithm, which solves a sequence of convex subproblems by linearizing $g$ while keeping  $f$ and $h$ unchanged. In recent years, several efficient algorithms have been developed,  including the proximal DC algorithm \cite{gotoh2018dc,wen2018proximal} and its enhanced versions  \cite{banert2019general,lu2019nonmonotone,aragon2020boosted,liu2022inexact,phan2024difference}. These algorithms, however, face challenges when additional nonconvex constraints of the form $\Ccal:=\{x\in\Rbb^n \mid c_i(x) \leq 0, i = 1, 2, \ldots, m\}$ are involved, where each $c_i(\cdot)$ is a smooth DC function. Two main strategies have been explored to address such constraints. The first leverages the exact penalty framework \cite{le2012exact}, which relies on establishing error bounds for specific types of constraints. Yet, it remains unclear whether such error bounds hold when $\Mcal$ is a Riemannian submanifold. The second strategy linearizes the concave part of each constraint at the current iterate \cite{yu2021convergence,zhang2023retraction}.   However, the convergence of such methods typically depends on the Mangasarian-Fromovitz constraint qualification (MFCQ), which generally fails in  the Riemannian setting, even in the simple case of the sphere \footnote{In their setting, the MFCQ requires that, for any $x \in \Ccal$, there exists $d \in \Rbb^n$ such that $\nabla c_i(x)^\tran d < 0$ for all $i$ with $c_i(x) = 0$. This condition clearly fails on the  sphere   $\Mcal = \{x\in\Rbb^n\mid x^\tran x = 1\}$, which corresponds to $c_1(x) = x^\tran x - 1$ and $c_2(x) = -x^\tran x + 1$.}.

\paragraph{Nonsmooth Riemannian optimization} Recently, nonsmooth Riemannian optimization has attracted growing attention.  For problem \eqref{equ:prob:dc} with $g(\cdot) = 0$, a variety of algorithms have been developed starting from the seminal ManPG method \cite{chen2020proximal}; see for instance \cite{chen2020alternating,xiao2021exact,huangwei2022riemannian,beck2023dynamic,peng2023riemannian,zhou2023semismooth,deng2024oracle,li2024riemannian,liu2024penalty,si2024riemannian,xu2024riemannian,xu2025oracle}.  For a comprehensive overview of recent advances, we refer the reader to \cite{chen2024nonsmooth} and the references therein.
When $g(\cdot) \neq 0$, only a few algorithms have been developed for Hadamard manifolds, such as Riemannian proximal point algorithms \cite{souza2015proximal,almeida2020modified} and Riemannian DC algorithms \cite{bergmann2024difference}. However,  these methods are inapplicable to many commonly used manifolds, such as the sphere and Stiefel manifolds, which are not Hadamard. 
 More recently, the work \cite{li2024proximal} extended ManPG to fractional and DC-structured problems via a proximal-gradient-subgradient method, but the requirement of exact subproblem solutions precludes an overall complexity guarantee. Table~\ref{table:complexity} summarizes existing algorithms that provide complexity guarantees for achieving an $\epsilon$-Riemannian critical point  of problem \eqref{equ:prob:dc} (see Definition~\ref{def:epsilon:critical}).  {In particular, RALM \cite{deng2024oracle}, RADMM \cite{li2024riemannian},  RSG \cite{beck2023dynamic,peng2023riemannian}, and OADMM \cite{yuan2024admm} consider general nonsmooth terms of the form $h(\mathcal{A}(x))$, where $\mathcal{A}$ is a linear operator, while RADA \cite{xu2024riemannian}  and RALM \cite{xu2025oracle}  further extend this capability to smooth, possibly a nonlinear operator $\mathcal{A}$. In contrast, the algorithms developed in this paper address DC-structured problems with $\mathcal{A}$ as the identity map, and extending the framework to general operators will be investigated in future work.}

\begin{table}[!t]
 \footnotesize
\setlength{\extrarowheight}{1.2pt} 
\centering
\caption{Complexity for achieving an $\epsilon$-Riemannian critical point.}  \label{table:complexity} 
\begin{tabular}{@{\hspace{1mm}}lcccc@{\hspace{1mm}}}
\hline
Algorithm & $h(\cdot) - g(\cdot)$ is DC & \# $\rgrad f(\cdot)$  &  \# $\Retr_x(\cdot)$  &   \# $\prox_h(\cdot)$  \\ 
\hline 
ManPG \cite{chen2020proximal}          & \ding{55}  & $\mathcal{O}(\epsilon^{-2})$ &  $\mathcal{O}(\epsilon^{-2})$ & --- \\
IRPG \cite{huang2023inexact,huang2025riemannian}    & \ding{55}   & $\mathcal{O}(\epsilon^{-2})$& $\mathcal{O}(\epsilon^{-2})$ & --- \\
SPLG \cite{liu2024penalty}           & \ding{55}  & $\mathcal{O}(\epsilon^{-2})$& $\mathcal{O}(\epsilon^{-2})$ & --- \\
\hdashline[0.5pt/2pt]
RALM \cite{deng2024oracle,xu2025oracle}         & \ding{55}   & $\mathcal{O}(\epsilon^{-3})$& $\mathcal{O}(\epsilon^{-3})$ & $\mathcal{O}(\epsilon^{-3})$ \\ 
RADMM \cite{li2024riemannian}              & \ding{55}   & $\mathcal{O}(\epsilon^{-4})$ & $\mathcal{O}(\epsilon^{-4})$& $\mathcal{O}(\epsilon^{-4})$ \\
RSG \cite{beck2023dynamic,peng2023riemannian}         & \ding{55} & $\mathcal{O}(\epsilon^{-3})$ & $\mathcal{O}(\epsilon^{-3})$& $\mathcal{O}(\epsilon^{-3})$ \\
RADA  \cite{xu2024riemannian}   &  \ding{55} & $\mathcal{O}(\epsilon^{-3})$ & $\mathcal{O}(\epsilon^{-3})$& $\mathcal{O}(\epsilon^{-3})$ \\
\hdashline[0.5pt/2pt]
OADMM \cite{yuan2024admm}             &  \ding{51} & $\mathcal{O}(\epsilon^{-3})$ & $\mathcal{O}(\epsilon^{-3})$& $\mathcal{O}(\epsilon^{-3})$ \\
\hdashline[0.5pt/2pt]
 iRPDC-BB [this work]                          & \ding{51}& ${\color{black}\mathcal{O}(\epsilon^{-2})}$ & ${\color{black}\mathcal{O}(\epsilon^{-2})}$& ${\color{black}\mathcal{O}(\epsilon^{-4})}$ \\
  iRPDC-NFG [this work]                          & \ding{51}  & ${\color{black}\mathcal{O}(\epsilon^{-2})}$ & ${\color{black}\mathcal{O}(\epsilon^{-2})}$& ${\color{black}\mathcal{O}(\epsilon^{-3} \log \epsilon^{-1})}$ \\
 iRPDC-AR [this work]                          & \ding{51}& ${\color{black}\mathcal{O}(\epsilon^{-2})}$ & ${\color{black}\mathcal{O}(\epsilon^{-2})}$& ${\color{black}\mathcal{O}(\epsilon^{-3})}$ \\
\hline
\end{tabular}
\end{table}

\subsection{Our contributions} 
While recent works such as \cite{yuan2024admm} and \cite{li2024proximal} have studied the nonsmooth Riemannian DC optimization problem \eqref{equ:prob:dc}, they have not explored its connection to sparse optimization. This paper bridges this gap by establishing this relationship and develops practical algorithms with both outer iteration and overall complexity guarantees for solving problem \eqref{equ:prob:dc}. Our main contributions are as follows:

(i) \emph{Equivalence between Riemannian DC and sparse models}: We show that the DC models \eqref{prob:sphere:l0:capped:l1} and \eqref{prob:sphere:knorm} are equivalent to their sparse counterparts \eqref{prob:sphere:ksparse:l0:regularized} and \eqref{prob:sphere:ksparse} over the sphere manifold, provided the DC parameters are sufficiently large. This is, to our knowledge, the first such equivalence result in the manifold setting, extending similar results from the Euclidean case \cite{le2015dc,gotoh2018dc,bian2020smoothing}.

(ii) \emph{Inexact Riemannian proximal DC (iRPDC) algorithmic  framework}: 
We propose an iRPDC algorithmic framework that incorporates the ManPG method~\cite{chen2020proximal} with the classical DC algorithm~\cite{tao1997convex}. 
A novel inexactness criterion is introduced for solving the subproblem, 
and it serves as the foundation for a linesearch procedure that adaptively captures the local curvature. 
Such a linesearch procedure has not been explicitly considered in existing inexact variants of ManPG~\cite{huang2023inexact,huang2025riemannian}.
We establish that the iRPDC algorithmic  framework attains an $\epsilon$-Riemannian critical point within $\mathcal{O}(\epsilon^{-2})$ iterations. 
When $g(\cdot) = 0$, our framework reduces to a new inexact variant of ManPG.

(iii) \emph{Practical algorithms with complexity guarantees}:  We develop three iRPDC algorithms, namely iRPDC-NFG, iRPDC-BB, and iRPDC-AR, based on different subproblem solvers.  A key feature of these algorithms is that the subproblem tolerance is determined from previous iterates, rather than the current (yet unavailable) one as required in existing methods. 
All three achieve $\mathcal{O}(\epsilon^{-2})$ outer iterations, with respective overall complexities of $\mathcal{O}(\epsilon^{-3}\log \epsilon^{-1})$,  $\mathcal{O}(\epsilon^{-4})$, and $\mathcal{O}(\epsilon^{-3})$. 
Even in the special case where $g(\cdot) = 0$, they lead to new inexact ManPG algorithms with guaranteed iteration complexity. 
A detailed comparison with existing methods is provided in Table~\ref{table:complexity}.

Finally, numerical results on SPCA demonstrate the effectiveness of the proposed Riemannian DC models and the efficiency of iRPDC algorithms. 

The rest of this paper is organized as follows. Section \ref{section:notations} introduces the notation and preliminaries. Section \ref{section:DC:sparse} discusses the equivalence between the DC models \eqref{prob:sphere:l0:capped:l1} and \eqref{prob:sphere:knorm} and their sparse optimization counterparts \eqref{prob:sphere:ksparse:l0:regularized} and \eqref{prob:sphere:ksparse}. Section \ref{section:RDCA} presents the proposed iRPDC algorithmic framework, followed by Section \ref{section:complexity}, which introduces the practical iRPDC algorithms and establishes its overall complexity. Section \ref{section:numerical} reports numerical results, and Section \ref{section:concluding} provides concluding remarks.

\section{Notation and preliminaries}  \label{section:notations}
 This section provides a brief review of the notation and preliminaries used in Riemannian optimization \cite{absil2009optimization,boumal2023introduction}. 
 For a smooth function $f: \Ecal \to \Rbb$, the Riemannian gradient at $x \in \Mcal$, where $\Mcal$ is a Riemannian submanifold  of $\Ecal$ endowed with the metric induced by the ambient space,  is the unique vector $\rgrad f(x)$ satisfying 
\be \label{equ:rgrad}
\iprod{\nabla f(x)}{\eta} = \iprod{\rgrad f(x)}{\eta}, \quad \forall\, \eta \in \tangent_x \Mcal,
\ee 
where $\tangent_x \Mcal$ denotes the tangent space of $\Mcal$ at $x$. 
It is given by  $\rgrad f(x) = \proj_{\tangent_x \Mcal}(\nabla f(x))$, where $\proj_{\tangent_x \Mcal}(\cdot)$  denotes the orthogonal projector onto  $\tangent_x \Mcal$.   For the Stiefel manifold $\stief = \{X \in \Rbb^{n \times r}\mid X^\tran X = I_r\}$, where $I_r$ is the $r$-by-$r$ identity matrix, we have 
$\tangent_X \Mcal = \{\eta \in \Rbb^{n \times r}\mid X^\tran \eta + \eta^\tran X = 0\}$ and 
$\proj_{\tangent_x \Mcal}(d) = d - X (X^\tran d + d^\tran X)/2$ for $d \in \Rbb^{n \times r}$.  We denote the unit sphere by $\Scal = \{x \in \Rbb^n \mid x^\tran x = 1\}$, which is a special Stiefel manifold with $r = 1$.  

For a convex function \(h: \Ecal \to \Rbb\),  let $\partial h(x)$ and $\partial_{\Rcal} h(x)$ denote the Euclidean and Riemannian subdifferential, respectively.  According to \cite[Theorem 5.1]{yang2014optimality},  we have 
\begin{equation}\label{equ:partial:h:R}
\partial_{\Rcal} h(x) = \proj_{\tangent_x \Mcal}(\partial h(x)).
\end{equation} 
 Moreover, for any given $\sigma > 0$, the Moreau envelope and proximal mapping of $h(\cdot)$ are defined by  $M_{\sigma h}(x)= \min_{u \in \Ecal} \big\{ h(u) + (2\sigma)^{-1} \|u - x\|^2 \big\}$ and $\prox_{\sigma h}(x)= \argmin_{u \in 
\Ecal} \big\{ h(u) +  (2\sigma)^{-1}  \|u - x\|^2 \big\}$, respectively.

 A retraction restricted to $\tangent_x \Mcal$ is a smooth mapping $\Retr_x: \tangent_x \Mcal \to \Mcal$, satisfying  (i) $\Retr_x(0_x) = x$, where $0_x$ is the origin of $\tangent_x \Mcal$; (ii) $\frac{\mathrm{d}}{\mathrm{d} t} \Retr_x(t \eta)|_{t = 0} = \eta$ for all $\eta \in \tangent_x \Mcal$. 
We assume that $\Retr_x(\cdot)$ is globally well-defined on $\tangent_x \Mcal$ and satisfies the following  properties \cite{boumal2019global,liu2019quadratic}. 
\begin{assumption}\label{assumption:retraction}
 There exist constants $\iota_1, \iota_2 > 0$ such that   
\be \label{equ:retraction:bound}
\|\Retr_x(\eta) - x\| \leq \iota_1 \|\eta\|, \quad  \|\Retr_x(\eta) - x - \eta \| \leq \iota_2 \|\eta\|^2, \quad \forall\,x \in \Mcal, \eta\in \tangent_x \Mcal.  
\ee
\end{assumption} 

The following result extends the first-order optimality from the Euclidean setting \cite{ahn2017difference,tao1997convex} to the Riemannian setting; see \cite[Theorems 4.1 and 5.1]{yang2014optimality}. 
\begin{lemma} \label{lemma:critical}
Let $\widehat x \in \Mcal$ be a local minimizer of problem \eqref{equ:prob:dc}.  Then, $\widehat x$ is  a Riemannian critical point, namely, $0 \in \grad f(\widehat x) + \partial_{\Rcal} h(\widehat x) - \partial_{\Rcal} g(\widehat x).$
\end{lemma}

Inspired by near-approximate stationarity concepts in  \cite{davis2019stochastic,li2024riemannian,tian2024no}, we define the notion of $\epsilon$-Riemannian critical point as follows. 
\begin{definition} \label{def:epsilon:critical}
We say that $x \in \Mcal$ is an $\epsilon$-Riemannian critical point of problem  \eqref{equ:prob:dc}  if  there exists a point $y \in \Ecal$ satisfying $\|y - x\|\leq \epsilon$ such that  
\be\label{equ:Delta}
\dist\big(0, \grad f(x) +  \partial_{\Rcal} h(y) - \partial_{\Rcal} g(x)\big)\leq \epsilon.  
\ee
\end{definition}

\section{Equivalence between DC and sparse models over manifolds}\label{section:DC:sparse}
This section establishes the equivalence between  DC formulations and sparse models over the manifold.  We begin by exploring the connection between the DC model \eqref{prob:sphere:l0:capped:l1} and the $\ell_0$-regularized model \eqref{prob:sphere:ksparse:l0:regularized}. 
By adapting the proof  techniques from \cite[Theorems 1\,\&\,2]{le2015dc}, we obtain the following results. 
\begin{theorem}
Let $\{\upsilon_t\} \subset \Rbb_+$ be a sequence with $\upsilon_t \to +\infty$ and $x_t^\star$ be a global (or local) minimizer of problem \eqref{prob:sphere:l0:capped:l1} corresponding to $\upsilon = \upsilon_t$.  Then, any accumulation point of  $\{x_t^\star\}$ is a global (or local) minimizer of problem  \eqref{prob:sphere:ksparse:l0:regularized}. 
\end{theorem}

While the above result holds asymptotically, we next show that exact equivalence holds on the sphere for a finite  $\upsilon$. To this end,  inspired by \cite[Lemma 2.3]{bian2020smoothing},  we first establish a lower bound property of the Riemannian critical points of problem \eqref{prob:sphere:l0:capped:l1}.   
\begin{lemma}\label{theorem:equiv:sparse:l0:regularized:00}  
Let $\Mcal = \Scal$ in problem \eqref{prob:sphere:l0:capped:l1} and let  $\bar x \in \Scal$ be a Riemannain critical point of problem \eqref{prob:sphere:l0:capped:l1}.  If  $\upsilon \geq {L_f^0}/{\sigma} + \sqrt{n}$, then for each index $i$, either $|\bar x_{(i)}| \geq 1/\upsilon$ or $\bar x_{(i)} = 0$. Consequently, $\Phi_{\capp}(\bar x) = \|\bar x\|_0$. 
\end{lemma}
\begin{proof}
The tangent space at $\bar x$ is  $\tangent_{\bar x} \Scal = \{d \in \Rbb^n \mid \bar x^\tran d = 0\}$, and the projection is  $\proj_{\tangent_{\bar x} \Scal}(\eta) = \eta - \iprods{\bar x}{\eta} \bar x $ for any $\eta \in \mathbb{R}^n$.  Since $\bar x$ is a Riemannian critical point of problem \eqref{prob:sphere:l0:capped:l1}, 
it follows from \eqref{equ:partial:h:R}, Lemma \ref{lemma:critical}, and  Example \ref{equ:prob:capped:l1} that 
\be\label{equ:lemma:dc:equiv:1:00}
0 = \grad f(\bar x) +  \widetilde \xi -  \iprods{\bar x}{\widetilde \xi} \bar x
\ee
for some  $\widetilde \xi \in \partial h(\bar x) - \partial g(\bar x)$ with $h(\bar x) =  \sigma\upsilon \|\bar x\|_1$ and 
 $g(\bar x) =  \sigma\sum_i \max\{\upsilon |\bar  x_{(i)}|-1, 0\}$. Suppose for contradiction that $0 < |\bar x_{(j)} | < 1/\upsilon$ for some $j$.  Then $|\widetilde \xi_{(j)}|  = \sigma \upsilon$. Since \(\|\bar{x}\| = 1\) and $|\widetilde \xi_{(i)}|\leq \sigma \upsilon$ for all $i$,  it follows that  $\iprods{\bar x}{\widetilde \xi}| \leq \|\bar x\| \cdot \|\widetilde \xi\| \leq \sigma \upsilon \sqrt{n}$.  This, together with \eqref{equ:lemma:dc:equiv:1:00}, leads to
\be\label{equ:lemma:dc:equiv:1:01}
|(\grad f(\bar x))_{(j)}| =  \big| \widetilde \xi_{(j)}  -\iprods{\bar x}{\widetilde \xi} \bar x_{(j)} \big| \\  
>   \sigma \upsilon (1  -    \sqrt{n}  |\bar x_{(j)}|) >\sigma(\upsilon - \sqrt{n}). 
\ee
On the other hand, since \( f \) is \(L_f^0\)-Lipschitz continuous, we have 
\be \label{equ:lemma:dc:equiv:1:02}
|(\grad f(\bar x))_{(j)}|\leq \|\grad f(\bar x)\| = \|\proj_{\tangent_{\bar x} \Scal} (\nabla f(\bar x)) \|\leq \|\nabla f(\bar x)\| \leq L_f^0. 
\ee
Combining \eqref{equ:lemma:dc:equiv:1:01} and \eqref{equ:lemma:dc:equiv:1:02} gives $\upsilon < L_f^0 / \sigma + \sqrt{n}$, contradicting the assumption. Thus, no such $j$ exists, and the result follows. In particular, this implies $\Phi_{\capp}(\bar x) = \|\bar x\|_0$.
\end{proof}

\begin{theorem}\label{theorem:equiv:sparse:l0:regularized}  
Let $\Mcal = \Scal$ in problems \eqref{prob:sphere:ksparse:l0:regularized} and \eqref{prob:sphere:l0:capped:l1}. If $\upsilon \geq {L_f^0}/{\sigma} + \sqrt{n}$, then the two problems share the same set of global minimizers. Moreover, any local minimizer of problem \eqref{prob:sphere:l0:capped:l1} is also a local minimizer of problem \eqref{prob:sphere:ksparse:l0:regularized}.  
\end{theorem}
\begin{proof} 
Let $x^\star \in \Scal$ and $x_{\upsilon}^\star \in \Scal$ be the global minimizers of problems \eqref{prob:sphere:ksparse:l0:regularized} and \eqref{prob:sphere:l0:capped:l1}, respectively.  
By the optimality of $x_{\upsilon}^\star$ and $x^\star$, and  the property $\Phi_{\capp}(x) \leq \|x\|_0$ for  any $x \in \mathbb{R}^n$, we have 
\[
f(x_\upsilon^\star) + \sigma \Phi_{\capp}(x_\upsilon^\star)  \leq f(x^\star) + \sigma \Phi_{\capp}(x^\star)  \leq f(x^\star) + \sigma \|x^\star\|_0  \leq f(x_\upsilon^\star) + \sigma \|x_\upsilon^\star\|_0.
\]
 Lemma \ref{theorem:equiv:sparse:l0:regularized:00} yields $\|x_\upsilon^\star\|_0 = \Phi_{\capp}(x_\upsilon^\star)$, so  
$f(x_\upsilon^\star) + \sigma \Phi_{\capp}(x_\upsilon^\star) =   f(x_\upsilon^\star) + \sigma \|x_\upsilon^\star\|_0,$
confirming both problems share identical global minimizers.

To prove the second claim, let $\tilde x$ be a local minimizer of \eqref{prob:sphere:l0:capped:l1}. Then, there exist a neighborhood $\mathcal{N}$ of $x$ such that $f(\tilde x) + \sigma \Phi_{\upsilon}(\tilde x) \leq 
f(x) + \sigma \Phi_{\upsilon}(x)$ for all $x \in \mathcal{N} \cap \Scal$.  By Lemmas \ref{lemma:critical} and  \ref{theorem:equiv:sparse:l0:regularized:00}, we have $\Phi_{\upsilon}(\tilde x) = \|\tilde x\|_0$. Moreover, since $\Phi_{\upsilon}(x) \leq \|x\|_0$ for any $x\in \mathbb{R}^n$,   it follows that  $f(\tilde x) + \sigma \|\tilde x\|_0\leq f(x) + \sigma \|x\|_0$ for all $x \in \mathcal{N}\cap \Scal$. This means that $\tilde x$ is also a local minimizer of problem \eqref{prob:sphere:ksparse:l0:regularized}.  This completes the proof.  
\end{proof}

Next, we show the equivalence between the DC model \eqref{prob:sphere:knorm} and the $\ell_0$-constrained model \eqref{prob:sphere:ksparse}, following an argument similar to \cite[Theorem 17.1]{nocedal1999numerical}.
\begin{theorem}
Let $\{\gamma_t\} \subset \Rbb_+$  with $\gamma_t \to +\infty$, and let $x_t^\star$ be a global minimizer of problem \eqref{prob:sphere:knorm}  with $\gamma = \gamma_t$. Then, any accumulation point of  $\{x_t^\star\}$ is a global minimizer of problem  \eqref{prob:sphere:ksparse}. Moreover, if each $x_t^*$ is a local minimizer and some accumulation point $x^*$ is feasible for problem  \eqref{prob:sphere:ksparse}, then $x^*$ is also a local minimizer of problem  \eqref{prob:sphere:ksparse}.
\end{theorem}

As in the previous case, the above equivalence holds only asymptotically. We now show that on the sphere manifold, a similar result can be obtained for a finite $\gamma$. As a first step, we establish a local error bound for the feasible set  $\mathcal{S}_k := \{x\in \Scal \mid \|x\|_0 \leq k\}$ of problem \eqref{prob:sphere:ksparse}.
\begin{lemma}\label{lemma:error:bound:sparse:knorm}
For any $x \in \Scal$, we have 
 \be \label{equ:error:bound}
 \dist(x, \mathcal{S}_k) \leq \sqrt{2}(1 + \sqrt{k/n})^{-1/2}\left(\|x\|_1 - \normmm{x}_k \right).
  \ee 
\end{lemma}
\begin{proof}
Without loss of generality, assume  $|x_{(1)}| \geq |x_{(2)}| \geq \cdots \geq  |x_{(n)}|$.  
  Let  $x_{1:k} =  (x_{(1)}, x_{(2)}, \ldots, x_{(k)})^\tran$. Since $x \in \Scal$, we have 
  \be \label{equ:hatx:bound:2}
  \begin{aligned} 
  1 - \|x_{1:k}\|^2 ={}&   |x_{(k + 1)}|^2 + |x_{(k+2)}|^2  + \cdots + |x_{(n)}|^2  \\  \leq{}&  ( |x_{(k + 1)}| +  |x_{(k+2)}| + \cdots + |x_{(n)}|)^2 = \left(\|x\|_1 - \normmm{x}_k\right)^2. 
  \end{aligned}
  \ee 
 Moreover, for $k+1 \leq i \leq n$, we have $|x_{(i)}|^2 \leq k^{-1} \|x_{1:k}\|^2$, so the first equality in \eqref{equ:hatx:bound:2} implies 
 $1 - \|x_{1:k}\|^2\leq (n-k)k^{-1} \|x_{1:k}\|^2$, which yileds   
 \be\label{equ:hatx:bound:3} 
   \|x_{1:k}\|^2 \geq k/n. 
 \ee
Let $\widetilde x := \proj_{\Scal_k}(x)=  \begin{pmatrix} x_{1:k}^\tran &  0\end{pmatrix}^\tran/\|x_{1:k}\| \in  \mathcal{S}_k$. Then,   
 $\dist(x, \mathcal{S}_k) ^2 = \|x - \widetilde x\|^2 =   2 (1 - \|x_{1:k}\|)  
 =  2(1 - \|x_{1:k}\|^2)/(1 + \|x_{1:k}\|)$,
which, together with \eqref{equ:hatx:bound:2} and \eqref{equ:hatx:bound:3}, yields the desired error bound \eqref{equ:error:bound}. 
\end{proof}
\begin{lemma}\label{equ:theorem:equiv:sparse:knorm:kkt}
Let $\Mcal = \Scal$ in problem \eqref{prob:sphere:knorm}, and let  $\bar x \in \Scal$ be a Riemannian critical point of problem \eqref{prob:sphere:knorm}.  If $\gamma >  n L_f^0/k$, then  $\bar x$ is $k$-sparse, i.e., $\|\bar x\|_1 - \normmm{\bar x}_k = 0$.
\end{lemma} 
\begin{proof}
Without loss of generality,  assume $|\bar x_{(1)}| \geq |\bar x_{(2)}| \geq \cdots \geq  |\bar x_{(n)}|$.
By the optimality condition as shown in Lemma \ref{lemma:critical}, we have $0 = \grad f(\bar x) + \gamma \widetilde \xi - \gamma \iprods{\bar x}{\widetilde \xi} \bar x$ for some $\widetilde \xi \in \partial \|\bar x\|_1 - \partial \normmm{\bar x}_k$. Suppose for contradiction that  $\|\bar x\|_1 - \normmm{\bar x}_k > 0$. Then, there exists an index $j \geq k + 1$ such that  $ |\bar x_{(j)}| > 0$. Let  $j^*$ be the largest such index. By definition, $\bar x_{i} = 0$ for  all $i \geq j^* + 1$ and $|\bar x_{i}|^2 \geq 1/j^*$ for $1 \leq i \leq k$. Consequently, $\sum_{i = k+1}^{j^*} |\bar x_{(i)}|^2 \leq 1 - k/j^* $. Also, note that  $\widetilde \xi_{i} = 0$ for  $1\leq i \leq k$, $|\widetilde \xi_i| = 1$ for  $k+1 \leq i \leq j^*.$ Hence,    
$|\iprods{\bar x}{\widetilde \xi} \bar x_{(j^*)}| \leq \sum_{i = k+1}^{j^*} |\bar x_{(i)}|^2  \leq 1-k/j^* \leq  1 - k/n$. Evaluating the $j^*$-th component of the optimality condition and using \eqref{equ:lemma:dc:equiv:1:02} yields
\[
L_f^0 \geq |\grad f(\bar x)_{(j^*)}| = \gamma | \widetilde \xi_{(j^*)} -  \iprods{\bar x}{\widetilde \xi} \bar x_{(j^*)}| \geq  \gamma (1 -  | \iprods{\bar x}{\widetilde \xi} \bar x_{(j^*)}|) \geq \gamma k/n. 
\]
This contradicts  $\gamma > nL_f^0/k$, and thus  $\bar{x}$ must be $k$-sparse.
\end{proof} 
\begin{theorem}\label{theorem:equiv:sparse:knorm}
Let $\Mcal = \Scal$ in problems \eqref{prob:sphere:ksparse} and \eqref{prob:sphere:knorm}. If  $\gamma >  n L_f^0/k$, then the two problems share the same set of global minimizers.  Moreover, any local minimizer of problem \eqref{prob:sphere:knorm}  is also a local minimizer of problem \eqref{prob:sphere:ksparse}. 
\end{theorem}
\begin{proof}
By noting that $n/k >  \sqrt{2}(1 + \sqrt{k/n})^{-1/2}$, the first claim follow directly from the error bound in Lemma~\ref{lemma:error:bound:sparse:knorm}, along with \cite[Lemmas 5\,\&\,9]{liu2024extreme} and \cite[Proposition~9.1.2]{cui2021modern}.
For the second claim, Lemmas \ref{lemma:critical} and \ref{equ:theorem:equiv:sparse:knorm:kkt} imply that any local minimizer of problem \eqref{prob:sphere:knorm} is feasible for problem \eqref{prob:sphere:ksparse} when $\gamma > n L_f^0 / k$. Applying \cite[Lemma~9]{liu2024extreme}, such a point is also a local minimizer of problem \eqref{prob:sphere:ksparse}.
\end{proof}

 Some remarks are in order. First, although the equivalence results in Theorem \ref{theorem:equiv:sparse:l0:regularized} and \ref{theorem:equiv:sparse:knorm} are established specifically on the sphere, they represent the first such results in the context of Riemannian DC optimization. Extending them to general manifolds remains an open question. Second, the error bound  \eqref{equ:error:bound} is of independent interest as it directly characterizes the $\ell_0$-constrained manifold set; see \cite{jiang2023exact,liu2024extreme,liu2024extreme2,chen2025tight} for recent developments.  Finally, the bound in \eqref{equ:error:bound} is tight. For instance, when  $n = 2$ and $k = 1$, and $x = (\sqrt{1/2}, \sqrt{1/2})^\tran$, we have $(1, 0)^\tran \in \proj_{\mathcal{S}_k}(x)$ and $\dist(x, \mathcal{S}_k)=  \sqrt{2 - \sqrt{2}} = \sqrt{2}(1 + \sqrt{k/n})^{-1/2}\left(\|x\|_1 - \normmm{x}_k \right)$, which exactly attains the bound.

\section{An iRPDC algorithmic framework}\label{section:RDCA}
In this section, we first present the Riemannian proximal DC algorithm (RPDCA) in Section~\ref{subsection:RPDCA}. Building on this, Section~\ref{subsection:iRPDCA:framework} introduces the proposed iRPDC algorithmic framework and establishes its iteration complexity for achieving an $\epsilon$-Riemannian critical point of problem~\eqref{equ:prob:dc}.
\subsection{RPDCA}\label{subsection:RPDCA}
For any $x \in \Mcal$, let $\widetilde \xi_x \in \partial g(x)$ be a subgradient, and define 
\be \label{equ:shorthand:0}
\xi_x = \proj_{\tangent_x \Mcal}(\widetilde \xi_x), ~  p_x = \rgrad f(x) - \xi_x,  ~  L_x :=  2 \iota_2 \big(  \|\nabla f(x) - \widetilde \xi_x\| + {L_h^0} \big) + \iota_1^2 L_f.
\ee
For notational simplicity, we define the following quantities at the iterate $x_j \in \Mcal$:
 \begin{equation}\label{equ:shorthand}
 \widetilde \xi_j: = \widetilde \xi_{x_j}, \quad \xi_j:=\xi_{x_j}, \quad p_j:=p_{x_j}, \quad L_j := L_{x_j}.
 \end{equation}

We begin with a key majorization result for the pullback $F  \circ \Retr_x: \tangent_x \Mcal \to \Rbb$, a concept introduced in \cite{boumal2019global}. This result plays a central role in our framework. 
\begin{lemma}\label{lemma:majorization}
Suppose that Assumptions \ref{assumption:f:h:g} and \ref{assumption:retraction} hold. Then, for any $x \in \Mcal$ and $\eta \in \tangent_x \Mcal$, we have  
\be \label{equ:F:descent:general:00} 
\begin{aligned} 
 F(\Retr_x(\eta))  \leq \iprod{p_x}{\eta} + \frac{L_x}{2}\|\eta\|^2 +  h(x + \eta)  + F(x) - h(x), 
\end{aligned}
\ee
where $L_x $ defined in \eqref{equ:shorthand:0} satisfies the uniform bound 
\be \label{equ:Lx:bound} 
L_x \leq L: = 2 \iota_2 (L_f^0  +  {L_g^0} + {L_h^0}) + \iota_1^2 L_f, \quad \forall\,x \in \Mcal. 
\ee 
\end{lemma}
\begin{proof}
We first bound $f(\Retr_x(\eta))$. From \eqref{equ:f:descent:euclidean}  and \eqref{equ:rgrad}, we have
\begin{multline}\label{equ:f:descent}
 f(\Retr_x(\eta)) 
\leq f(x)  +   \iprod{\nabla f(x)}{\Retr_x(\eta) - x - \eta} \\
 +   \iprod{\rgrad f(x)}{\eta}  + \frac{L_f}{2} \|\Retr_x(\eta) - x\|^2.  
\end{multline}
For $h(\Retr_x(\eta))$, since $h$ is convex and ${L_h^0}$-Lipschitz, 
  we have 
\be
\label{equ:h:descent}
\begin{aligned}
h(\Retr_x(\eta)) \leq{}&  h(x + \eta) + {L_h^0} \|\Retr_x(\eta) - x - \eta\|. 
\end{aligned}
\ee    
Next,  since $\xi_{x} = \proj_{\tangent_x \Mcal}(\widetilde \xi_x)$ as given in \eqref{equ:shorthand:0}, it holds that  $\iprods{\widetilde \xi_x}{\eta} = \iprod{\xi_x}{\eta}$ for any $\eta \in \tangent_x \Mcal$.  By the convexity of $g(\cdot)$ and the inclusion \( \widetilde \xi_x \in \partial g(x) \), we obtain
\be 
\begin{aligned}
\label{equ:g:descent}
 g(\Retr_x(\eta))   
\geq g(x)  +    \iprods{\widetilde \xi_x}{\Retr_x(\eta) - x - \eta} +    \iprods{ \xi_x}{\eta}. 
\end{aligned}
\ee
 Combining \eqref{equ:f:descent}, \eqref{equ:h:descent},  \eqref{equ:g:descent}, and using the definitions of $p_x$ and $L_x$ in  \eqref{equ:shorthand:0} and the property \eqref{equ:retraction:bound}, we 
 obtain  the desired  \eqref{equ:F:descent:general:00}.
 
 Noting that  $\iprods{\nabla f(x) - \widetilde \xi_x}{\Retr_x(\eta) - x - \eta} \leq \|\nabla f(x) - \widetilde \xi_x\| \cdot \|\Retr_x(\eta) - x - \eta \|.$  In addition, since $f$ and $g$ are $L_f^0$- and $L_g^0$-Lipschitz continuous, it follows directly that \eqref{equ:Lx:bound} holds. The proof is complete. 
 \end{proof}

The inequality \eqref{equ:F:descent:general:00} forms the foundation for designing RPDCA.    
Specifically, at  the iterate \(x_j \in \Mcal\),  we  choose $\ell_j$ as an estimate of $L_j$,  since the parameters $\iota_1$, $\iota_2$, $L_h^0$, and $L_f$ may be unavailable or overestimated in practice. We require that  
\be \label{equ:Lk}
L_{\min} \leq \ell_j \leq L_{\max},
\ee 
 where $L_{\max} \geq L_{\min} > 0$ are prescribed constants. To update $x_{j + 1}$, we solve the subproblem 
\begin{align} \label{equ:etak}
 \min_{\eta \in \tangent_{x_j} \Mcal}\, \left\{ q_j(\eta):= \iprod{p_j}{\eta} + \frac{\ell_j}{2}\|\eta\|^2 +  h(x_j + \eta)\right\}
\end{align} 
to obtain the search direction $\eta^\star_j:= \argmin_{\eta \in \tangent_{x_j} \Mcal}\, q_j(\eta)$.  
By the optimality of $\eta_j^\star$, we have the  following {\it sufficient decrease property}: 
 \be \label{equ:eta:opt} 
 q_j(\eta_j^\star) \leq q_j(0)   -  \frac{\ell_j}{2} \|\eta_j^\star\|^2.
 \ee
Then, similar to Lemma~\ref{lemma:F:descent:inexact}, we can obtain the descent estimate   
 \[
 F\!\left(\Retr_{x_j}(\tau \eta^\star_j)\right) \leq  F(x_j) -  \frac{2 - L_j\ell_j^{-1}\tau}{2}  \ell_j \tau \|\eta_j^\star\|^2, \quad \forall\,\tau \in [0,1].
 \] 
By choosing a suitable stepsize $\tau_j$ (uniformly  bounded away from zero),  the RPDCA update is given by
\be\label{equ:RpDCA}
x_{j + 1} =  \Retr_{x_j}(\tau_j \eta^\star_j),
\ee
which ensures  
$F(x_{j+1}) \leq F(x_j) -  c \tau_j \ell_j \|\eta_j^\star\|^2$ for some given constant $c \in (0,1)$. 
In analogy with Theorem \ref{theorem:outer:complexity},  any limit point of the sequence $\{x_j\}$ generated by RPDCA \eqref{equ:RpDCA} is a Riemannian critical point of problem \eqref{equ:prob:dc}. {Moreover, RPDCA attains an $\epsilon$-Riemannian critical point of problem \eqref{equ:prob:dc} within $\mathcal{O}(\epsilon^{-2})$ iterations.} Notably, RPDCA reduces to  ManPG proposed by \cite{chen2020proximal} when $g(\cdot) = 0$ and $c = 1/2$. 
\subsection{The iRPDC framework and its complexity}\label{subsection:iRPDCA:framework}Since computing $\eta^\star_j$ exactly may be unnecessary or computationally expensive in practice, we introduce an \textit{inexact} Riemannian proximal DC (iRPDC) algorithmic framework. It allows an approximate solution $\eta_j \in \tangent_{x_j} \Mcal$ to \eqref{equ:etak}, while preserving key properties required for convergence analysis.  
 To compute an $\epsilon$-Riemannian critical point of problem \eqref{equ:prob:dc},  we define the accuracy  parameter
 \be \label{equ:epsilon:j}
\epsilon_j = \min\{\ell_j^{-1},1\}\epsilon. 
\ee 
Given constants $\rho \in [0,1)$ and $c \in (0,1-\rho/2)$, we then require that the direction $\eta_j$  satisfies the following inexact conditions: 
\begin{subequations}
 \label{equ:eta:inexact}
\begin{align}
&q_j(\eta_j) \leq q_j(0)  -  \frac{(1 - \rho)\ell_j}{2} \|\eta_j\|^2   + \mu_j + c \beta_1 \ell_j \epsilon_j^2,  \label{equ:eta:inexact:a} \\ 
& \|\eta_j^\star\|\leq \kappa \|\eta_j\|   + (\chi_j +  \beta_2\epsilon_j^2)^{1/2},\label{equ:eta:inexact:b}
\end{align}
where the parameters satisfy  
\begin{align}
&\kappa > 0, \quad {\beta_1 > 0},  \quad\beta_2 \geq 0,  \quad 2(\beta_1 \kappa^2 + \beta_2) < 1,\label{equ:beta1:beta2} \\
&\mu_j \geq  - \frac{\rho \ell_j}{2} \|\eta_j\|^2,  \quad\chi_j \geq 0, \quad  \sum_{t = 0}^{j} \chi_t \leq \chi,  \label{equ:chi} 
\end{align}
\end{subequations}
where $\mu \geq 0$ and $\chi >0$ are some constants.  

Intuitively, condition \eqref{equ:eta:inexact:a} ensures a \textit{controllable sufficient decrease} in the model function $q_j(\cdot)$, extending the exact case \eqref{equ:eta:opt}. Moreover, the optimality condition of subproblem~\eqref{equ:etak} implies that if $\|\eta_j^\star\|\leq \epsilon_j$, then by~\eqref{equ:epsilon:j} and~\eqref{equ:Delta}, the iterate $x_j$ is already an $\epsilon$-Riemannian critical point of problem~\eqref{equ:prob:dc}. However, since $\eta_j^\star$ is unavailable in practice, the condition $\|\eta_j^\star\|\leq \epsilon_j$ cannot be verified directly. Instead, condition~\eqref{equ:eta:inexact:b} provides a computable upper bound for $\|\eta_j^\star\|$ in terms of  the implementable quantity  $\|\eta_j\|$ and a summable error sequence, ensuring that $\|\eta_j^\star\|$ is small whenever $\|\eta_j\|$ is small. This yields a practical criterion for approximate stationarity and supports the convergence analysis. Clearly, the exact solution $\eta_j^\star$ satisfies \eqref{equ:eta:inexact} trivially with $\kappa=1$, $\rho=\beta_1=\beta_2=0$, and $\chi_j\equiv \mu_j\equiv 0$. Practical strategies for computing such $\eta_j$ will be given in Section~\ref{section:complexity}.

The following lemma establishes a controlled descent property for $F(\cdot)$.   
\begin{lemma}\label{lemma:F:descent:inexact}
Let $\eta_j$ satisfy condition \eqref{equ:eta:inexact}. Then, for any $\tau \in [0,1]$,  
\be\label{equ:F:descent:inexact}
 F(\Retr_{x_j}(\tau \eta_j)) \leq F(x_j)  -\frac{2- \rho - L_j\ell_j^{-1} \tau}{2} \ell_j \tau \|\eta_j\|^2  + \tau (\mu_j + c \beta_1 \ell_j \varepsilon_j^2). 
 \ee
\end{lemma}
\begin{proof}
By \eqref{equ:shorthand}, \eqref{equ:F:descent:general:00}, and  \eqref{equ:etak},   for any $\tau \in [0,1]$,  we have 
\be \label{equ:F:descent:proof:a00} 
 F(\Retr_{x_j}(\tau \eta_j)) 
\leq  F(x_j)  + q_j(\tau \eta_j) - q_j(0) + \frac{L_j - \ell_j}{2} \tau^2 \|\eta_j\|^2. 
\ee
By the convexity of $h$, it holds that $h(x_j  + \tau \eta_j) \leq \tau h(x_j + \eta_j) + (1 - \tau) h(x_j)$, which, together with the definition of $q_j(\cdot)$ in \eqref{equ:etak}, gives 
\[
q_j(\tau \eta_j) \leq \tau (q_j(\eta_j) - q_j(0)) + q_j(0) + \frac{\tau^2 - \tau}{2} \ell_j \|\eta_j\|^2.
\]
Substituting this into \eqref{equ:F:descent:proof:a00} and applying \eqref{equ:eta:inexact:a} gives \eqref{equ:F:descent:inexact}. 
\end{proof}

 Once such \(\eta_j\) is obtained, we perform backtracking to ensure a controllable sufficient decrease. Given a contraction parameter $s \in (0,1)$, we select  the smallest nonnegative integer $i$ such that $\tau_j = s^i$  
satisfies 
\be  \label{equ:F:descent:abstract}
F(\Retr_{x_j}(\tau_j \eta_j))   
\leq   F(x_j) -  c  \tau_j \ell_j \|\eta_j\|^2  +   \tau_j (\mu_j + c \beta_1 \ell_j  \epsilon_j^2). 
\ee
We then set 
\be\label{equ:iRPDC:x:j+1}
x_{j+1} = \Retr_{x_j}(\tau_j \eta_j).
\ee

Since $\ell_j \geq L_{\min}$ by \eqref{equ:Lk} and $L_j \leq L$ by \eqref{equ:Lx:bound}, we have $(2- \rho - L_j\ell_j^{-1} \tau)/2 \geq c$ whenever $\tau \leq\min\{(2 - \rho - 2c)L_{\min}/L, 1\}$.  Thus,   the backtracking procedure terminates in a finite number of steps, and the resulting stepsize $\tau_j$ is uniformly bounded away from zero. These facts are formalized below. 
 \begin{lemma} \label{lemma:abstract:descent}
Suppose that Assumptions \ref{assumption:f:h:g} and \ref{assumption:retraction} hold, and that the inexactness conditions in \eqref{equ:eta:inexact} are satisfied. Let $\bar \tau := \min\{(2 - \rho - 2c)L_{\min}/L, 1\}$. Then, the backtracking index $i$ in Line 6 of Algorithm \ref{alg:rdca} satisfies 
\be \label{equ:i:max}
i \leq  i_{\max}:= \lceil\log_s \bar \tau\,\rceil \quad \text{and}\quad \tau_j \geq \min\{s \bar \tau, 1\}. 
\ee 
 Moreover, inequality~\eqref{equ:F:descent:abstract} holds for all  $j \geq 0$.   
\end{lemma}

\DecMargin{8pt}   
\begin{algorithm2e}[!t]
\DontPrintSemicolon
\KwIn{$\epsilon > 0$, $x_0 \in \Mcal$, $\rho \in [0, 1)$, $c \in (0, 1- \rho/2)$, $s \in (0,1), \beta_1 > 0$, $\beta_2\in [0, 1/2 - \beta_1 \kappa^2)$, $0 < L_{\min} \leq L_{\max}$.}
\For{$j = 0, 1, \dots$}{
     Choose $\ell_j \in [L_{\min}, L_{\max}]$ and select $\mu$ satisfying \eqref{equ:muj}.\;
    Solve the subproblem \eqref{equ:etak} inexactly to obtain $\eta_j \in \tangent_{x_j} \Mcal$ satisfying \eqref{equ:eta:inexact}.\; 
   \textbf{if} $\kappa \|\eta_j\|  +  (\chi_j +  \beta_2\epsilon_j^2)^{1/2} \leq \epsilon_j$\, \textbf{then}~return~$x_j$.\;
   \For{$i = 0,1,\ldots$}{
    Set $\tau_j = s^i$ and update $x_{j+1} = \Retr_{x_j}(\tau_j\,\eta_j)$.\;
     \textbf{if} \eqref{equ:F:descent:abstract} holds \textbf{then}~break.
      }
}
\caption{An iRPDC algorithmic framework for solving problem \eqref{equ:prob:dc}}\label{alg:rdca}
\end{algorithm2e}

 The complete iRPDC algorithmic framework is summarized in  Algorithm \ref{alg:rdca}.     To guarantee convergence, we further assume that the sequence $\{\mu_j\}$ satisfies
 \be \label{equ:muj}
\sum_{t = 0}^{j} \tau_t \mu_t \leq   \mu, \quad \forall\,j \geq 0.
\ee
Practical strategies for constructing such  $\{\mu_j\}$  will be discussed in Section \ref{section:complexity}. 

We now present our main convergence and iteration complexity results. \begin{theorem} \label{theorem:outer:complexity}
 Suppose that Assumptions \ref{assumption:f:h:g} and \ref{assumption:retraction} hold, and that the inexactness conditions in \eqref{equ:eta:inexact} and \eqref{equ:muj} are satisfied.
Let $\{x_j\}$ be the sequence generated by  Algorithm \ref{alg:rdca}. If $\epsilon> 0$, then Algorithm  \ref{alg:rdca}  terminates within $\mathcal{O}(\epsilon^{-2})$ iterations and returns an $\epsilon$-Riemannian critical point of problem \eqref{equ:prob:dc}. 
\end{theorem}
\begin{proof}
For any $J_1 \geq 0$, from \eqref{equ:iRPDC:x:j+1},  summing  \eqref{equ:F:descent:abstract} over $j = 0, 1, \ldots, J_1$ and applying the bound on $\mu_j$ from  \eqref{equ:muj} yields 
 \be \label{equ:outer:complexity:new:00} 
 \sum_{j = 0}^{J_1}  \tau_j \ell_j \|\eta_j\|^2   \leq    c^{-1}(F(x_0) -  F^\star   +  \mu) + \beta_1 \sum_{j = 0}^{J_1}  \tau_j \ell_j \epsilon_j^2,
 \ee
 where $F^\star$ denotes the optimal value of problem \eqref{equ:prob:dc}.   Using $L_{\min}\leq \ell_j\leq  L_{\max}$ and $0 < \tau_j\leq 1$, it follows from \eqref{equ:chi} that 
 \be\label{equ:outer:complexity:new:01}  
 \begin{aligned}
 \sum_{j = 0}^{J_1}  \tau_j \ell_j (\chi_j + \beta_2 \epsilon_j^2) \leq \chi L_{\max}  + \beta_2 \sum_{j = 0}^{J_1}  \tau_j \ell_j \epsilon_j^2. 
 \end{aligned}
 \ee
Multiplying \eqref{equ:outer:complexity:new:00} by $\kappa^2$ and adding \eqref{equ:outer:complexity:new:01},   
we apply the inequality $(a+b)^2 \leq 2(a^2 + b^2)$ with $a = \kappa \|\eta_j\|$ and $b =(\chi_j + \beta_2 \epsilon_j^2)^{1/2}$ to obtain
 \be \label{equ:outer:complexity:new:03}
 \begin{aligned}
 \sum_{j = 0}^{J_1}  \tau_j \ell_j\left(\kappa \|\eta_j\| + (\chi_j + \beta_2 \epsilon_j^2)^{1/2}\right)^2
 \leq  C_1 + 2(\beta_1 \kappa^2 + \beta_2) \sum_{j = 0}^{J_1}  \tau_j \ell_j  \epsilon_j^2,
 \end{aligned}
 \ee
 where $C_1:= 2\kappa^2 c^{-1}(F(x_0) -  F^\star   +  \mu) + 2\chi L_{\max}$.

Suppose for contradiction that the algorithm does not terminate. Then
 $\kappa \|\eta_j\| + (\chi_j + \beta_2 \epsilon_j^2)^{1/2} > \epsilon_j$  for all \( j \geq 0 \). Substituting this into \eqref{equ:outer:complexity:new:03}, and using $2(\beta_1 \kappa^2 + \beta_2)< 1$ by \eqref{equ:beta1:beta2}, we deduce
 $\sum_{j = 0}^{J_1}  \tau_j \ell_j  \epsilon_j^2 \leq  (1 -2 \beta_1 \kappa^2 - 2\beta_2)^{-1} C_1$ for all $J_1 \geq 0.$
Considering that $\tau_j \geq \min\{s \bar \tau, 1\}$ by \eqref{equ:i:max}, $\ell_j \geq L_{\min}$, and $\epsilon_j \geq \min\{L_{\max}^{-1},1\} \epsilon$ by $\ell_j \geq L_{\max}$ and \eqref{equ:epsilon:j}, this makes a contradiction. Therefore, the algorithm must terminate after finitely many iterations. Let $J \geq 1$ be the termination index (the case  $J = 0$ is trivial). Then,  we have 
\be \label{equ:etak:etak:star:003}  
\kappa \|\eta_{J}\| + (\chi_{J} + \beta_2 \epsilon_J^2)^{1/2} \leq \epsilon_J, \quad \kappa \|\eta_j\| + (\chi_j + \beta_2 \epsilon_j^2)^{1/2} > \epsilon_j, \,  j = 0,1, \ldots, J-1.
\ee 
Substituting \eqref{equ:etak:etak:star:003}  into \eqref{equ:outer:complexity:new:03}, and using $0< \tau_j \leq 1$ together with \eqref{equ:beta1:beta2}, we have 
 \be \label{equ:outer:complexity:new:05}
\sum_{j=0}^{J-1}  \tau_j \ell_j \epsilon_j^2 \leq  (1 - 2\beta_1 \kappa^2 - 2\beta_2)^{-1}(C_1 +  \ell_J \epsilon_J^2). 
\ee
Using \eqref{equ:epsilon:j}, together with $\tau_j \geq \min\{s \bar \tau, 1\}$ by \eqref{equ:i:max} and $L_{\min}\leq \ell_j \leq L_{\max}$, we further have
 $\ell_J \epsilon_J^2 \leq \epsilon^2$ and $ \tau_j \ell_j \epsilon_j^2 \geq  \min\{L_{\max}^{-1}, L_{\min}\} \min\{s \bar \tau, 1\}\epsilon^2$. Substituting these into \eqref{equ:outer:complexity:new:05} gives the iteration bound
\be \label{equ:barJ:bound:new} 
J \leq  \frac{C_1 +  \epsilon^2}{(1 - 2 \beta_1 \kappa^2 - 2\beta_2)\min\{s \bar \tau, 1\} \cdot \min\{L_{\max}^{-1}, L_{\min}\}} \epsilon^{-2}. 
\ee

It remains to verify that  $x_{J}$ is an $\epsilon$-Riemannian critical point as defined in \eqref{equ:Delta}.       
From \eqref{equ:eta:inexact:b} and \eqref{equ:etak:etak:star:003}, we have
$\|\eta_{J}^\star\| \leq \epsilon_J \leq \epsilon$. 
Moreover,  {the optimality of \eqref{equ:etak} at iteration $j = J$ implies that there exists $y = x_{J} + \eta_{J}^\star$ such that $0 \in \grad f(x_{J}) - \xi_{J} + \ell_{J} \eta_{J}^\star + \partial_{\Rcal} h(y)$, where $\xi_J \in \partial_{\Rcal} g(x_J)$. Since  $\ell_{J}\|\eta_{J}^\star\| \leq \ell_J \epsilon_J \leq \epsilon$ by \eqref{equ:epsilon:j},  it follows from  \eqref{equ:Delta}  that $x_{J}$ is an $\epsilon$-Riemannian critical point  of problem \eqref{equ:prob:dc}.}
Combining this with \eqref{equ:barJ:bound:new},  we conclude that the algorithm terminates within $\mathcal{O}(\epsilon^{-2})$ iterations and returns an $\epsilon$-Riemannian critical point of problem \eqref{equ:prob:dc}. 
 \end{proof}
\begin{remark}
If $\epsilon = 0$, then $\epsilon_j = 0$ by \eqref{equ:epsilon:j}, and it follows from \eqref{equ:F:descent:abstract} and \eqref{equ:muj} that $F(x_j) \leq F(x_0) + \mu$ for all $j \geq 1$.  By Assumption \ref{assumption:f:h:g}-(iii), the sequence $\{x_j\}$ is bounded and thus has a limit point. Following standard arguments (e.g., \cite[Theorem 3.1]{huang2022extension}),  any such limit point is a Riemannian critical point of problem \eqref{equ:prob:dc}.
\end{remark}

We conclude this section with some remarks on condition \eqref{equ:eta:inexact}.  First, condition \eqref{equ:eta:inexact:a} differs from the inexactness criteria based on  the $\varepsilon$-subdifferential, $\varepsilon$-optimality, or their 
variants (see, e.g., \cite{villa2013accelerated,drusvyatskiy2019efficiency,zheng2024new,yang2025inexact} and the references therein
for some recent advances). Instead, it directly compares $q_j(\eta_j)$ and $q_j(0)$ and explicitly permits a degree of nonmonotonicity via introducing the error term $\mu_j$. This generalization distinguishes our framework from existing inexact Riemannian proximal gradient methods for the special case $g(\cdot) = 0$ (e.g., \cite{huang2023inexact,huang2025riemannian}), where \(q_j(\eta_j) \leq q_j(0)\) is typically required. Second,  allowing such nonmonotonicity in $q_j(\eta_j)$ increases flexibility 
in choosing  \(\mu_j\), which can be adapted using the information from previous iterates; see  Section \ref{section:complexity} for practical strategies. Third, condition \eqref{equ:eta:inexact} provides the foundation for establishing the linesearch condition \eqref{equ:F:descent:abstract}, which exploits an adaptive estimate of the local curvature and may further enhances the practical performance of our framework. It should be noted that the original ManPG~\cite{chen2020proximal} incorporates a linesearch procedure, but it requires exact solutions of the subproblems. In contrast, a linesearch strategy built upon inexact subproblem criteria, as enabled by condition \eqref{equ:eta:inexact}, has not been considered in the inexact Riemannian proximal gradient methods~\cite{huang2023inexact,huang2025riemannian}.

\section{iRPDC algorithms and complexity analysis}\label{section:complexity}
In Section~\ref{subsection:practical:condition}, we discuss practical strategies for selecting the parameters in conditions \eqref{equ:eta:inexact} and \eqref{equ:muj} by analyzing the dual of the subproblem \eqref{equ:etak} and establishing several useful properties. Then, in Section~\ref{subsection:iRPDCA}, we present and analyze several implementations of the iRPDC algorithmic framework, including both practically effecient and theoretically motivated algorithms. 
\subsection{Practical implementation of conditions \eqref{equ:eta:inexact} and \eqref{equ:muj}} \label{subsection:practical:condition}
To this end, we first consider the dual formulation of the subproblem \eqref{equ:etak} and derive several key properties that will be instrumental in the design and analysis of our practical implementation.

Since $\Ecal$  is a finite-dimensional Euclidean space, we denote it by $\Rbb^n$ for simplicity. Let $d$ be the dimension of $\Mcal$. For any $x_j \in \Mcal$,  the tangent space $\tangent_{x_j} \Mcal$ can be characterized by  
\be\label{equ:tangent:space:expression} 
\tangent_{x_j} \Mcal =\{\eta \in \Rbb^n \mid B_{j}^{\tran}\eta = 0\},
\ee
where the columns of $B_{j} \in \Rbb^{n \times (n-d)}$ form an orthonormal basis of the normal space $\tangent_{x_j} \Mcal^\perp$, so that $B_j^\tran B_j = I_{n-d}$. Computing $B_j$ is efficient for many common manifolds, such as the Stiefel manifold,  the Grassmann manifold, and the fixed-rank matrix manifold; see \cite{huang2023inexact} for details. 

Using \eqref{equ:tangent:space:expression}, we can equivalently reformulate problem \eqref{equ:etak} as 
\begin{equation} \label{equ:etak:simple}
 \min\limits_{\eta \in \Rbb^n}\ q_j(\eta)\quad  \st\quad   B_j^\tran \eta = 0. 
\end{equation} 
Let $\lambda \in \Rbb^{n-d}$ be the Lagrange multiplier associated with the linear constraint $B_j^{\tran}\eta = 0$.  The dual problem of \eqref{equ:etak:simple}, in the minimization form, is given by 
\be \label{equ:prob:etak:dual:00} 
\min_{\lambda \in \Rbb^{n-d}}\, \left\{\psi_j(\lambda):= -\min_{\eta \in \Rbb^n}\, \Big\{q_j(\eta) + \iprod{\lambda}{B_j^\tran \eta}\!\Big\}\right\}. 
\ee
For any fixed $\lambda \in \Rbb^{n - d}$, the inner minimization problem in \eqref{equ:prob:etak:dual:00} admits a unique solution
\be\label{equ:eta:lambda}  
\eta_j(\lambda) =  \prox_{\frac{h}{\ell_j}}\left(x_j -   \frac{1}{\ell_j}(p_j   +  B_j \lambda)\right) - x_j.
\ee
Since \(p_j \in \tangent_{x_j}\Mcal\) (see \eqref{equ:shorthand}), it follows from \eqref{equ:tangent:space:expression} that \(B_j^\tran p_j = 0\). A direct calculation yields the dual formulation of the subproblem \eqref{equ:etak}  as
\be \label{equ:prob:etak:dual} 
\min_{\lambda \in \Rbb^{n-d}}\, \left\{ \psi_j(\lambda) = \frac{1}{2\ell_j} \|\lambda\|^{2}   - M_{\frac{h}{\ell_j}}\left(x_j -   \frac{1}{\ell_j}(p_j   +  B_j \lambda)\right)  + \frac{1}{2\ell_j} \|p_j\|^2\right\}. 
\ee

\begin{proposition}\label{prop:psi}
The gradient $\nabla \psi$ is $\ell_j^{-1}$-Lipschitz continuous, and
\be \label{equ:nabla:theta:grad} 
\nabla \psi_j(\lambda) = - B_j^\tran  \eta_j(\lambda). 
\ee
Moreover, if  $\psi_j(\lambda) \leq \psi_j(0)$, then $\| \lambda\| \leq  2 {L_h^0}$.
\end{proposition}
\begin{proof}
The Lipschitz continuity of $\nabla \psi$ and the identity \eqref{equ:nabla:theta:grad} follow  from \cite[Theorems 6.42 and 6.60]{beck2017first} and \eqref{equ:eta:lambda}.  To show $\|\lambda\|\leq 2L_h^0$, note that $\psi_j(\lambda)\leq \psi_j(0)$ implies
\[
   \frac{1}{2\ell_j} \|\lambda\|^2 
\leq   M_{\frac{h}{\ell_j}}\left(x_j -   \frac{1}{\ell_j}(p_j   +  B_j \lambda)\right) - M_{\frac{h}{\ell_j}}\left(x_j -   \frac{1}{\ell_j} p_j\right) 
\leq \frac{{L_h^0}}{\ell_j} \| \lambda\|, 
\]
where the second inequality follows from the ${L_h^0}$-Lipschitz continuity of $M_{h/\ell_j}(\cdot)$ \cite[Lemma 2.1]{drusvyatskiy2019efficiency}.  The claim follows.
\end{proof}

For any $\lambda\in \Rbb^{n - d}$, define
\be \label{equ:def:widehat:eta:lambda}
\widehat \eta_j(\lambda) = \proj_{\tangent_{x_j} \Mcal} \big(\eta_j(\lambda)\big) \in \tangent_{x_j} \Mcal.
\ee  
The following lemma shows that $\widehat \eta_j(\lambda)$ \revgpt{can potentially satisfy} the inexact condition \eqref{equ:eta:inexact}
by appropriately controlling  $\|\nabla \psi_j(\lambda)\|$. 
\begin{lemma} \label{lemma:basic:descent}
Let  $\lambda\in \Rbb^{n - d}$. Then, we have
\begin{subequations}
\begin{align} 
&q_j(\widehat \eta_j(\lambda)) \leq q_j(0) - \frac{\ell_j}{2} \|\widehat \eta_j(\lambda)\|^2 + 2{L_h^0} \|\nabla \psi_j(\lambda)\|\label{equ:qj:descent}\\
&\|\eta_j^\star\| \leq  \|\widehat \eta_j(\lambda)\|   + \big(4{L_h^0} \ell_j^{-1}\|\nabla \psi_j(\lambda)\| +   \|\nabla \psi_j(\lambda)\|^2\big)^{1/2}. \label{equ:etaj:star:bound}
\end{align}
\end{subequations}
\end{lemma}
\begin{proof} 
For simplicity, we drop the subscript $j$ in the proof.  We first prove  \eqref{equ:qj:descent}.    
 By the property of the proximal operator and \eqref{equ:eta:lambda}, there exists  $\zeta \in \partial h(x + \eta(\lambda))$ such that
\be \label{equ:h:xj:etalambda:opt} 
\ell \eta(\lambda) + p + B \lambda + \zeta = 0.
\ee
Using the convexity of $h(\cdot)$ at $x+\eta(\lambda)$, we have 
\be \label{equ:hxj:bound}
h(x) \geq h\bigl(x+\eta(\lambda)\bigr) - \iprod{\eta(\lambda)}{\zeta} = h\bigl(x+\eta(\lambda)\bigr) + \iprod{\widehat \eta(\lambda)- \eta(\lambda)}{\zeta} - \iprod{\widehat{\eta}(\lambda)}{\zeta}.
\ee
  Since $h$ is ${L_h^0}$-Lipschitz continuous,  we also have 
  \be  \label{equ:hxj:bound:02} 
 h(x + \eta(\lambda))  \geq h(x + \widehat \eta(\lambda)) - {L_h^0} \|\widehat \eta(\lambda)  -   \eta(\lambda) \|.
 \ee
 Moreover, noting that $\zeta\in\partial h(x+\eta(\lambda))$ and that $h(\cdot)$ is $L_h^0$-Lipschitz continuous, we obtain  
 \be \label{equ:xi:bound}
 \|\zeta\|\leq {L_h^0},
 \ee
 which implies 
 \be  \label{equ:hxj:bound:00}
  \iprod{\widehat \eta(\lambda) - \eta(\lambda)}{\zeta} \geq - {L_h^0} \|\widehat \eta(\lambda) - \eta(\lambda)\|. 
 \ee
Using the definition of $\widehat \eta(\lambda)$ in \eqref{equ:def:widehat:eta:lambda}, the expression for the tangent space in \eqref{equ:tangent:space:expression}, and  the property of the projection operator, we have 
 \be \label{equ:widehat:eta:lamdba:property:1} 
 \widehat \eta(\lambda) =   \eta( \lambda)- B  B^\tran \eta( \lambda), \quad  \iprod{\widehat \eta(\lambda)}{\eta(\lambda)} = \|\widehat \eta(\lambda)\|^2. 
 \ee
Since $B^\tran \widehat \eta(\lambda) = 0$ by $\widehat \eta(\lambda) \in \tangent_{x} \Mcal$, combining \eqref{equ:h:xj:etalambda:opt} and  \eqref{equ:widehat:eta:lamdba:property:1}, we have
 \be  \label{equ:hxj:bound:01}
 \iprod{\widehat \eta(\lambda)}{\zeta} = - \ell \|\widehat \eta(\lambda)\|^2  - \iprod{\widehat \eta(\lambda)}{p}. 
 \ee
 From \eqref{equ:widehat:eta:lamdba:property:1} and \eqref{equ:nabla:theta:grad}, it holds that 
\be \label{equ:hxj:bound:03}
 \|\widehat \eta(\lambda) - \eta(\lambda)\|  =\|B  B^\tran \eta( \lambda)\| = \|B^\tran \eta( \lambda)\| =  \|\nabla \psi(\lambda)\|. 
\ee 
 Substituting \eqref{equ:hxj:bound:02}, \eqref{equ:hxj:bound:00}, and \eqref{equ:hxj:bound:01} into \eqref{equ:hxj:bound}, and using  \eqref{equ:hxj:bound:03},  we get 
\[
 h(x) \geq h(x + \widehat \eta(\lambda))   - 2{L_h^0}  \| \nabla \psi(\lambda)\|   + \ell \|\widehat \eta(\lambda)\|^2   + \iprod{\widehat \eta(\lambda)}{p}. 
 \]
 Using the definition of $q(\cdot)$ in \eqref{equ:etak}, this inequality implies  \eqref{equ:qj:descent}.

We now prove \eqref{equ:etaj:star:bound}.   
 Let $\lambda^\star$ be an optimal solution of problem \eqref{equ:prob:etak:dual}, then $\eta^\star := \eta(\lambda^\star)$ is an optimal solution of problem \eqref{equ:etak:simple}. 
By strong duality, we have
\be\label{equ:eta:star:k:duailty} 
\psi(\lambda^\star) = - q(\eta^\star).
\ee
By \eqref{equ:prob:etak:dual:00}, we also have
$\psi( \lambda) = -q(\eta(  \lambda)) - \iprods{\eta(  \lambda)}{ B \lambda}$, 
which, together with \eqref{equ:eta:star:k:duailty},  $\psi(\lambda^\star) \leq \psi(\lambda)$, and \eqref{equ:h:xj:etalambda:opt}, implies   
\[
 q(\eta(\lambda)) - q(\eta^\star) 
\leq - \iprods{\eta(  \lambda)}{ B \lambda} 
\leq \iprods{\eta( \lambda)}{\zeta} + \iprods{\eta( \lambda)}{p} + \ell \|\eta(\lambda)\|^2.
\]
Using the definition of $q(\cdot)$ in \eqref{equ:etak},  we have   
\[  
\begin{aligned}
 & q(\widehat \eta(\lambda))  - q(\eta( \lambda))\\
={}&   \iprods{\widehat \eta(\lambda)-  \eta( \lambda)}{p}  + h(x +  \widehat \eta(\lambda) ) - h(x + \eta( \lambda))   + \frac{\ell}{2}(\|\widehat \eta(\lambda)\|^2 -  \| \eta(\lambda)\|^2)\\
\leq{}&  -\iprods{\widehat \eta(\lambda)}{\zeta}  -  \iprods{ \eta( \lambda)}{p}  + {L_h^0} \|  \widehat \eta(\lambda) - \eta(\lambda)\| -  \frac{\ell}{2}(\|\widehat \eta(\lambda)\|^2 + \| \eta(\lambda)\|^2),
\end{aligned}
\]
where the inequality uses \eqref{equ:hxj:bound:01} and  the $L_h^0$-Lipschitz continuity of $h(\cdot)$.   
Adding the two inequalities, and noting $\|\eta(\lambda)\|^2 - \|\widehat \eta(\lambda)\|^2 = \|\eta(\lambda) - \widehat \eta(\lambda)\|^2$ by \eqref{equ:widehat:eta:lamdba:property:1}, we have 
\be\label{equ:hxj:bound:04}
q(\widehat \eta(\lambda)) - q(\eta^\star) \leq \iprods{\eta(\lambda) - \widehat \eta(\lambda)}{\zeta} + {L_h^0} \|  \widehat \eta(\lambda) - \eta(\lambda)\|  + \frac{\ell}{2}\left(\| \eta(\lambda) - \widehat  \eta(\lambda)\|^2\right).
\ee
Since $q(\cdot)$ is $\ell$-strongly convex over $\tangent_{x}\Mcal$, we have 
$q(\widehat \eta(\lambda)) - q(\eta^\star) \geq (\ell/2) \|\widehat \eta(\lambda) - \eta^\star\|^2$. Combining this with \eqref{equ:xi:bound},   \eqref{equ:hxj:bound:03}, and
\eqref{equ:hxj:bound:04},  we derive the desired \eqref{equ:etaj:star:bound}.  
\end{proof}

It is worth mentioning that 
\cite[Lemma 5]{huang2025riemannian} established the following upper bound:
\[
q_j(\widehat \eta_j(\lambda)) \leq q_j(0) + (2L_h^0 + (\ell_j/2) \|\nabla \psi_j(\lambda)\|) \|\nabla \psi_j(\lambda)\|, 
\]
which, however, does not guarantee a decrease in $q_j(\cdot)$ and is thus insufficient for analyzing iteration complexity. In contrast, our bound in \eqref{equ:qj:descent} ensures a sufficient descent by well controlling  $\|\nabla \psi_j(\lambda)\|$ and is therefore stronger.

By Lemma \ref{lemma:basic:descent}, if we choose $\widetilde \lambda$ such that 
 \be  \label{equ:nabla:theta:condition:00} 
 \|\nabla \psi_j( \widetilde \lambda)\| \leq \varepsilon_j: =\min\left\{\frac{1}{2 L_h^0}  \left(\mu_j + \frac{\rho\ell_j}{2} \|\eta_j\|^2 + c \beta_1 \ell_j \epsilon_j^2\right), \frac{4 L_h^0}{\ell_j}\right\},
\ee
and set $\eta_j = \widehat \eta_j(\widetilde \lambda)$, 
then \eqref{equ:eta:inexact:a} is satisfied, and \eqref{equ:eta:inexact:b} holds with 
\be\label{equ:kappa:chij:beta2:practical}
\kappa = 1, \quad \chi_j  = \frac{4}{\ell_j}  \left(\mu_j + \frac{\rho\ell_j}{2} \|\eta_j\|^2 \right),\quad  \beta_2 = 4c\beta_1. 
\ee
In view of \eqref{equ:beta1:beta2}, this requires $\beta_1$ to satisfy  $2(1 + 4c)\beta_1 < 1$. The threshold $\varepsilon_j$ in \eqref{equ:nabla:theta:condition:00} is defined as the minimum of two terms, and the cap $4 L_h^0\ell_j^{-1}$ plays a crucial role.  Without this cap, $\varepsilon_j$ would be determined solely by the first term, which can exceed $4L_h^0\ell_j^{-1}$.  In that case, by Proposition \ref{equ:nabla:theta:grad}, any $\widetilde \lambda$ satisfying $\psi_j(\widetilde \lambda) \leq \psi_j(0)$ (e.g., $\widetilde \lambda = 0$) would automatically satisfy 
 \[
 \|\nabla \psi_j(\widetilde \lambda)\| = \|\nabla \psi_j(\widetilde \lambda) - \nabla \psi_j(\lambda_j^\star)\| \leq \ell_j^{-1}\|\widetilde \lambda - \lambda_j^\star\|\leq 4L_h^0 \ell_j^{-1} < \varepsilon_j.
 \] 
 making condition~\eqref{equ:nabla:theta:condition:00} trivially satisfied and potentially causing premature termination.

It remains to choose $\mu_j $ satisfying \eqref{equ:muj} such that $\chi_j$ in \eqref{equ:kappa:chij:beta2:practical} satisfies \eqref{equ:chi}.  Let  $\{\omega_j\}$ be a nonnegative and summable sequence, i.e., $\omega_j \geq 0$ and $\sum_{j=0}^{+\infty} \omega_j < +\infty$.  A simple summable choice is
\[
 \omega_j = \omega_0  \ell_{j}  (j + 1)^{-a},
 \]
where $\omega_0 \geq 0$ and $a > 1$ are constants. Based on this, we choose
\be \label{equ:mu:k:general}
 \mu_j =
\frac{\rho}{2}\left(\tau_{j-1} \ell_{j-1} \|\eta_{j-1}\|^2 -  \ell_j \|\eta_{j}\|^2\right) +   \omega_0  \ell_{j}  (j + 1)^{-a}, \quad \forall\, j \geq 0, 
\ee
with initialization $\eta_{-1} = 0$, $\tau_{-1} = 1$,  and  $\ell_{-1} \in [L_{\min}, L_{\max}]$.  
\revgpt{Since} $0 < \tau_j \leq 1$,  we have  
\be \label{equ:tjmuj} 
\tau_j \mu_j \leq \frac{\rho}{2}\left(\tau_{j-1} \ell_{j-1} \|\eta_{j-1}\|^2 -  \tau_j \ell_j \|\eta_{j}\|^2\right) +  \omega_0  \ell_{j}  (j + 1)^{-a},
\ee
ensuring that $\{\mu_j\}$ satisfies  \eqref{equ:muj} with $\mu = L_{\max} \omega_0 \sum_{j = 0}^{+\infty} (j+1)^{-a}$.  Under this choice, $\chi_j$ in \eqref{equ:kappa:chij:beta2:practical} becomes 
\be \label{equ:widetilde:chij} 
\chi_j = \frac{2\rho\tau_{j-1}   \ell_{j-1}  \|\eta_{j-1}\|^2 +  4 \omega_0 \ell_j  (j + 1)^{-a}}{\ell_j}. 
\ee
Since $\eta_{j-1}$ is known at the $j$-th iteration, and by an argument similar  to \eqref{equ:outer:complexity:new:00},  we have that,  for  $0 \leq j \leq J \leq \mathcal{O}(\epsilon^{-2})$ (with $J$ given in \eqref{equ:barJ:bound:new}), 
\[
 \sum_{t = 0}^{j} \tau_{t-1} \ell_{t-1} \|\eta_{t-1}\|^2   \leq    c^{-1}(F(x_0) -  F^\star   +  \mu) + \beta_1 \sum_{t = 0}^J  \tau_t  \ell_t  \epsilon_t^2 \leq C_2, 
\]
where $C_2$ is some constant, with the second inequality from $0<\tau_t\leq 1$ and \eqref{equ:epsilon:j}. Hence, the constructed $\chi_j$ in \eqref{equ:widetilde:chij} satisfies $\sum_{t=0}^j \chi_t \leq \chi$ for some constant $\chi$.

Building on the preceding discussions, in particular the choices of $\mu_j$ in \eqref{equ:mu:k:general} and $\chi_j$ in \eqref{equ:widetilde:chij}, the inexactness threshold $\varepsilon_j$ in \eqref{equ:nabla:theta:condition:00} admits the implementable form stated below.
 \begin{proposition} \label{prop:nabla:theta:condition:practical}
Let $\widetilde{\lambda} \in \mathbb{R}^{n-d}$ satisfy 
 \be \label{equ:nabla:theta:condition} 
\|\nabla \psi_j(\widetilde{\lambda})\| \leq \varepsilon_j
:= \min\left\{\frac{\rho \tau_{j-1}  \ell_{j-1}   \|\eta_{j-1}\|^2 + 2 \omega_0 \ell_j  (j + 1)^{-a}+ 2 c\beta_1 \ell_j \epsilon_j^2}{4L_h^0},\ \frac{4L_h^0}{\ell_j}\right\}.
\ee
Let $\eta_j = \widehat{\eta}_j(\widetilde{\lambda})$.
If $2(1 + 4c)\beta_1 < 1$, then conditions \eqref{equ:eta:inexact} and \eqref{equ:muj} hold with $\kappa = 1$ and $\beta_2 = 4c\beta_1$.
\end{proposition}
 
Finally, define the augmented function $F_{\rho}(x_j): = F(x_j) + \frac{\rho \tau_{j-1} \ell_{j-1}}{2} \|\eta_{j-1}\|^2$. By \eqref{equ:tjmuj},  the linesearch condition \eqref{equ:F:descent:abstract}  then implies the following relaxed form (sufficient for convergence analysis): 
 \be \label{equ:ls:new}
F_{\rho}(x_{j+1})   \leq F_{\rho}(x_{j})  - c  \tau_j \ell_j \|\eta_j\|^2 + c \beta_1  \tau_j \ell_j  \epsilon_j^2 +  \omega_0  \ell_{j}  (j + 1)^{-a},
\ee
which will be adopted in the practical algorithms described in  Section \ref{subsection:iRPDCA}.

\subsection{iRPDC: algorithms and complexity}\label{subsection:iRPDCA}
To compute a point  satisfying \eqref{equ:nabla:theta:condition}, we first consider two first-order approaches: (i) applying Nesterov's fast gradient (NFG) method to a regularized dual problem \cite{nesterov2018lectures}, and (ii) applying the safeguard BB gradient method \cite{barzilai1988two,dai2005projected} to the original dual problem. These lead to two {practically efficient algorithms within the iRPDC algorithmic framework}, denoted by iRPDC-NFG and iRPDC-BB, which are described below.  

We begin with iRPDC-NFG, which solves the following  regularized  dual problem:  
\be \label{prob:theta:regularization}
\min_{\lambda \in \Rbb^{n-d}}\, \Big\{\psi_{\delta_j}(\lambda):=\psi_j(\lambda) + \frac{\delta_j}{2} \|\lambda\|^{2}\Big\},
\ee
where the regularization parameter is $\delta_j :=  \varepsilon_j/(4{L_h^0})$, motivated by the upper bound $2 L_h^0$  in Proposition \ref{prop:psi}.   
The gradient $\nabla \psi_{\delta_j}$ is $(\ell_j^{-1} + \delta_j)$-Lipschitz continuous.  Starting from $\lambda^{(0)} = \lambda^{(-1)}=  0$, for $t = 0, 1, \ldots$,  the NFG method iterates as 
\be
\label{equ:nesterov}
\begin{cases}
y^{(t)}= \lambda^{(t)} + \frac{\sqrt{\kappa_j} - 1}{\sqrt{\kappa_j} + 1}\big(\lambda^{(t)}- \lambda^{(t-1)}\big),\\[4pt]
\lambda^{(t+1)}=   y^{(t)} - (\ell_j^{-1} + \delta_j)^{-1} \nabla \psi_{\delta_j}(y^{(t)}),
\end{cases}
\ee
where  $\kappa_j = 1 + (\ell_j \delta_j)^{-1}$.  

\begin{lemma} \label{lemma:inner:complexity}  
Suppose that Assumption \ref{assumption:f:h:g} holds. Then, the NFG method \eqref{equ:nesterov} returns  a point $\widetilde \lambda$  satisfying  \eqref{equ:nabla:theta:condition} in $\mathcal{O}\big(\varepsilon_j^{-1/2} \log (1 + \varepsilon_j^{-1})\big)$ iterations. 
Let $\eta_j = \widehat{\eta}_j(\widetilde{\lambda})$.
If $2(1 + 4c)\beta_1 < 1$, then conditions \eqref{equ:eta:inexact} and \eqref{equ:muj} hold with $\kappa = 1$ and $\beta_2 = 4c\beta_1$.
 \end{lemma}
 \begin{proof}
By Proposition \ref{prop:nabla:theta:condition:practical}, it suffices to establish the complexity of computing $\widetilde \lambda$ such that \eqref{equ:nabla:theta:condition} holds. Let $\lambda_{\delta_j}^\star$ be the unique minimizer of problem \eqref{prob:theta:regularization}. Similar to Proposition \ref{prop:psi}, we have  $\|\lambda^\star_{\delta_j}\|\leq 2 {L_h^0}.$
 Using \cite[Section 2.2.2]{nesterov2018lectures} and the choice  $\delta_j = \varepsilon_j/{(4{L_h^0})}$,  we obtain 
$\| \nabla \psi_j(\widetilde \lambda) \| \leq \varepsilon_j/2 + \ell_j^{-1} (8{L_h^0}\varepsilon_j^{-1} ( \psi_{\delta_j}(\widetilde \lambda) -   \psi_{\delta_j}( \lambda_{\delta_j}^\star)))^{1/2}.$
Hence, to ensure $\| \nabla \psi_j(\widetilde \lambda) \| \leq \varepsilon_j$, it suffices to require 
 $\psi_{\delta_j}(\widetilde \lambda) -   \psi_{\delta_j}(\lambda_{\delta_j}^\star) \leq (\ell_j^{2} \varepsilon_j^3)/(32 {L_h^0})$.
Since  $\nabla \psi_j$ is  $\ell_j^{-1}$-Lipschitz continuous, \cite[Theorem 2.2.7]{nesterov2018lectures} guarantees that this can be achieved within at most 
$3 \big\lceil (1 + {4{L_h^0} \ell_j^{-1}}{\varepsilon_j^{-1}})^{1/2} \log(1 +  {4 {L_h^0} \ell_j^{-1}}{ \varepsilon_j^{-1}} )\big\rceil$
iterations.
\end{proof}

\begin{algorithm2e}[!t]
\DontPrintSemicolon
\KwIn{$\epsilon > 0$, $x_0 \in \Mcal$, $\rho \in [0, 1), c \in (0, 1- \rho/2), s \in (0,1), \beta_1 \in (0, 1/(2 + 8c))$, $\omega_0 > 0$, $a > 1$, $0 < L_{\min} \leq L_{\max}$.
}
\For{$j = 0, 1, \dots$}{
     Choose $\ell_j \in [L_{\min}, L_{\max}]$ and select $\mu_j$ according to \eqref{equ:mu:k:general}.\;
    Use  the NFG method \eqref{equ:nesterov} to find a point $\widetilde \lambda$  satisfying \eqref{equ:nabla:theta:condition}.\; 
   Compute  $\eta_j$  as in Lemma \ref{lemma:inner:complexity}, and $\chi_j$ via \eqref{equ:widetilde:chij}.\;
   \textbf{if} $\|\eta_j\|  +  (\chi_j+  4c \beta_1\epsilon_j^2)^{1/2}\leq \epsilon_j$\, \textbf{then}~return~$x_j$.\;
   \For{$i = 0,1,\ldots$}{
    Set $\tau_j = s^i$ and update $x_{j+1} = \Retr_{x_j}(\tau_j\,\eta_j)$.\;
     \textbf{if} \eqref{equ:ls:new} holds \textbf{then}~break.
      }
}
\caption{A practical iRPDC-NFG for solving problem \eqref{equ:prob:dc}}\label{alg:rdca:2}
\end{algorithm2e}

We summarize the complete  iRPDC-NFG in Algorithm~\ref{alg:rdca:2}. Its iteration complexity is stated below. 
\begin{theorem}\label{theorem:overall:nfg}
Suppose that Assumptions~\ref{assumption:f:h:g} and~\ref{assumption:retraction} hold. Then, for any $0 < \epsilon \ll 1$, Algorithm~\ref{alg:rdca:2} returns an $\epsilon$-Riemannian critical point of problem \eqref{equ:prob:dc}  within $\mathcal{O}(\epsilon^{-2})$ outer iterations and $\mathcal{O}(\epsilon^{-3} \log \epsilon^{-1})$ inner iterations.  Therefore,  the algorithm requires $\mathcal{O}(\epsilon^{-2})$ evaluations of $\rgrad f(\cdot)$ and $\Retr_x(\cdot)$, and $\mathcal{O}(\epsilon^{-3} \log \epsilon^{-1})$ evaluations of $\prox_{h}(\cdot)$.  
\end{theorem}
\begin{proof}
The outer iteration bound follows from Theorem~\ref{theorem:outer:complexity}, whose conditions are ensured by Proposition~\ref{prop:nabla:theta:condition:practical}. Let $J \leq \mathcal{O}(\epsilon^{-2})$ be the total number of outer iterations. Since $\ell_j \in [L_{\min}, L_{\max}]$ and $\varepsilon_j = \mathcal{O}(\epsilon^{-2})$ (by \eqref{equ:epsilon:j} and  \eqref{equ:nabla:theta:condition}), Lemma~\ref{lemma:inner:complexity} shows that each inner subproblem requires at most $\mathcal{O}(\epsilon^{-1} \log \epsilon^{-1})$ inner iterations. Summing over $j$ from $0$ to $J$ yields the stated total inner complexity. The evaluation bounds then follow  directly from the algorithm structure.
\end{proof}

We next present iRPDC-BB, an alternative practical algorithm that applies the safeguard BB method \cite{barzilai1988two,dai2005projected} to solve the original dual subproblem \eqref{equ:prob:etak:dual}.
Given  constants  $0 <\varrho_2 < 1 < \varrho_1$, the  method starts from  $\lambda^{(0)} = 0$ and iterates as 
 \be \label{equ:method:BB}
 \lambda^{(t+1)} = \lambda^{(t)} - \nu_t\nabla \psi_j (\lambda^{(t)}),\quad \nu_t = \min\{\nu^{\mathsf{BB}}_{t}, \varrho_1 \ell_j\}2^{-m},
  \ee
 where   $\nu^{\mathsf{BB}}_{0} = \ell_j$ and $\nu^{\mathsf{BB}}_{t} = \| \lambda^{(t)} - \lambda^{(t-1)}\|^2/\langle{\lambda^{(t)} - \lambda^{(t-1)}},{\nabla \psi_j(\lambda^{(t)}) - \nabla \psi_j(\lambda^{(t-1)})}\rangle$ for $t \geq 1$ 
with the convention $0/0 = +\infty$. Here, $m$ is  the smallest nonnegative integer such that 
$\psi_j(\lambda^{(t+1)}) \leq \psi_j ( \lambda^{(t)}) -  \varrho_2\nu_t \|\nabla \psi_j (\lambda^{(t)})\|^2$. 
It is known that  $0 \leq m \leq \lceil \log_2 (\varrho_1/(2(1-\varrho_2)) \rceil$ and that  the method achieves the iteration complexity $\mathcal{O}(\varepsilon_j^{-1})$ for computing a point satisfying \eqref{equ:nabla:theta:condition} \cite[Theorem 10.26]{beck2017first}.  The complete iRPDC-BB algorithm is given in Algorithm~\ref{alg:rdca:bb}, and its iteration complexity is summarized below; the proof is similar to that of Theorem \ref{theorem:overall:nfg} and is omitted for brevity.
\begin{theorem}\label{theorem:overall:bb}
Suppose that Assumptions~\ref{assumption:f:h:g} and~\ref{assumption:retraction} hold. Then, for any $0 < \epsilon \ll 1$, Algorithm~\ref{alg:rdca:bb} returns an $\epsilon$-Riemannian critical point of problem \eqref{equ:prob:dc}  within $\mathcal{O}(\epsilon^{-2})$ outer  iterations and $\mathcal{O}(\epsilon^{-4})$ inner iterations. Therefore, the algorithm requires $\mathcal{O}(\epsilon^{-2})$ evaluations of $\rgrad f(\cdot)$ and $\Retr_x(\cdot)$, and $\mathcal{O}(\epsilon^{-4})$ evaluations of $\prox_{h}(\cdot)$.  
\end{theorem}

\begin{algorithm2e}[!t]
\DontPrintSemicolon
\KwIn{Same as in Algorithm~\ref{alg:rdca:2}, with additional parameters $0 < \varrho_2 < 1 < \varrho_1$.}
\For{$j = 0, 1, \dots$}{
    Same as in Algorithm~\ref{alg:rdca:2}, except that replacing the NFG method \eqref{equ:nesterov} with the safeguard BB method \eqref{equ:method:BB} to compute $\widetilde \lambda$ satisfying \eqref{equ:nabla:theta:condition}.\;
}
\caption{A practical iRPDC-BB  for solving problem \eqref{equ:prob:dc}}\label{alg:rdca:bb}
\end{algorithm2e}
 
We now present the third algorithm, iRPDC-AR, which improves the complexity of iRPDC-NFG by removing the $\log \epsilon^{-1}$ factor.
 To compute a point $\widetilde \lambda$ satisfying \eqref{equ:nabla:theta:condition},   this algorithm adopts the accumulative regularization (AR) method recently developed in \cite{lan2023optimal},  which enhances Nesterov's regularization technique \eqref{prob:theta:regularization}.  
 
 The method initializes with $\lambda^{(0)} = \lambda_0 = \bar \lambda_0 = 0$ and sets $\delta_{j,0} = \varepsilon_j / (8L_h^0)$.  For $i = 0, 1, \ldots,  \lceil \log_4 (2L_h^0\ell_j^{-1}\varepsilon_j^{-1})\rceil$,  it first updates the proximal center as 
 \be \label{equ:am:proximal:center}
 \bar \lambda_i = 0.25 \bar \lambda_{i-1} + 0.75 \lambda_{i-1},
 \ee and then solves the $i$-th AR subproblem 
\be \label{prob:theta:regularization:am}
\min_{\lambda \in \Rbb^{n-d}}\, \Big\{\psi_{\delta_{j,i}}(\lambda):=\psi_j(\lambda) + \frac{\delta_{j,i}}{2} \|\lambda - \bar \lambda_i\|^{2}\Big\}\quad \mbox{with} \quad \delta_{j,i} = 4^i \delta_{j,0}, 
\ee
by applying Nesterov's accelerated gradient method \cite{nesterov2018lectures} or FISTA \cite{beck2017first}  for $T_i :=  \left\lceil 16(\ell_j / \delta_{j,i} + 1)^{1/2} \right\rceil$ iterations.  Starting from $\lambda^{(0)} =\lambda^{(-1)}=  \lambda_{i-1}$,  the  updates are 
\be \label{equ:am:fista}
\begin{cases}
y^{(t)}= \lambda^{(t)} + \frac{t-1}{t+2}\big(\lambda^{(t)}- \lambda^{(t-1)}\big), \\[4pt]
\lambda^{(t+1)}=  y^{(t)} - (\ell_j^{-1} + \delta_{j,i})^{-1}\nabla \psi_{\delta_{j,i}}(y^{(t)}), 
\end{cases}
 \quad t = 0,1, \ldots, T_i. 
\ee
 The approximate solution to \eqref{prob:theta:regularization:am} is  then set as 
\be \label{prob:theta:regularization:lambda:i}
\lambda_i := \lambda^{(T_i+1)}.
\ee

\begin{algorithm2e}[!t]
\DontPrintSemicolon
\KwIn{Same as in Algorithm~\ref{alg:rdca:2}.}
\For{$j = 0, 1, \dots$}{
           Same as in Algorithm~\ref{alg:rdca:2}, except that replacing the NFG method \eqref{equ:nesterov} with the AR method \eqref{equ:am:proximal:center}, \eqref{equ:am:fista}, and \eqref{prob:theta:regularization:lambda:i} to compute $\widetilde \lambda$ satisfying \eqref{equ:nabla:theta:condition}.\;
}
\caption{A theoretical iRPDC-AR for solving problem \eqref{equ:prob:dc}}\label{alg:rdca:arm}
\end{algorithm2e}

According to \cite[Theorem 2.1]{lan2023optimal} and \cite[Theorem 10.34]{beck2017first}, this method produces a point $\widetilde \lambda$ satisfying  \eqref{equ:nabla:theta:condition}  within $\mathcal{O}(\varepsilon_j^{-1})$ iterations.  The resulting algorithm, iRPDC-AR, is summarized in Algorithm~\ref{alg:rdca:arm}, and its complexity is given below.
\begin{theorem}\label{theorem:overall}  Suppose that Assumptions~\ref{assumption:f:h:g} and~\ref{assumption:retraction} hold. Then, for any $0 < \epsilon \ll 1$, Algorithm~\ref{alg:rdca:arm} returns an $\epsilon$-Riemannian critical point of problem \eqref{equ:prob:dc}  within $\mathcal{O}(\epsilon^{-2})$ outer iterations and $\mathcal{O}(\epsilon^{-3})$ inner iterations. Therefore,  the algorithm requires $\mathcal{O}(\epsilon^{-2})$ evaluations of $\rgrad f(\cdot)$ and $\Retr_x(\cdot)$, and $\mathcal{O}(\epsilon^{-3})$ evaluations of $\prox_{h}(\cdot)$.  
\end{theorem}
\begin{remark}
Among the three algorithms, iRPDC-AR attains the best theoretical complexity, matching the best-known bound summarized in Table \ref{table:complexity}. However, it requires a fixed number of iterations to solve each AR subproblem, which may result in unnecessary computations in practice.  In contrast, the first two algorithms, iRPDC-NFG and iRPDC-BB, adaptively terminate the inner iterations based on the gradient norm, offering more efficient performance in practical applications. In addition to these first-order methods, semismooth Newton-type approaches \cite{li2018highly,xiao2018regularized} provide another viable option for solving the dual subproblem to obtain a point satisfying \eqref{equ:nabla:theta:condition}. First adopted in \cite{chen2020proximal}, such methods have been utilized  in  subsequent works \cite{chen2020alternating,huang2022extension,huang2023inexact} to address  \eqref{equ:prob:etak:dual} for problem \eqref{equ:prob:dc} with  $g(\cdot) = 0$.  Although they often demonstrate excellent empirical performance, their iteration complexity and superlinear convergence remain unclear in our specific setting.  
\end{remark}

We conclude this section with a few remarks.  First, when \(g(\cdot) = 0\) in \eqref{equ:prob:dc}, our proposed iRPDC algorithms reduce to  inexact versions of the ManPG method proposed by \cite{chen2020proximal}. {While several inexact variants of ManPG have been studied in this setting (e.g., \cite{chen2020alternating,huang2025riemannian,huang2023inexact}), to the best of our knowledge,  iRPDC is among the first algorithms in this line of work with a provable overall complexity.}  Second,  both iRPDC-NFG and iRPDC-AR improve upon existing methods listed in Table  \ref{table:complexity}. Specifically, these methods require \(\mathcal{O}(\epsilon^{-3})\) evaluations of  \(\rgrad f(\cdot)\), along with retractions and proximal mappings of \(h(\cdot)\). 
In contrast, iRPDC-NFG and iRPDC-AR both reduce the number of Riemannian gradient and retraction evaluations to $\mathcal{O}(\epsilon^{-2})$, with the number of proximal mappings being $\mathcal{O}(\epsilon^{-3}\log \epsilon^{-1})$ and $\mathcal{O}(\epsilon^{-3})$, respectively. This improvement can be particularly advantageous when evaluating $\rgrad f(\cdot)$ is computationally expensive (e.g., \cite{wang2022Riemannian}).

\section{Numerical results}\label{section:numerical}
In this section, we present numerical results on SPCA problems to evaluate the modeling effectiveness of the DC formulations \eqref{prob:sphere:l0:capped:l1} and \eqref{prob:sphere:knorm}, and the computational efficiency of the proposed iRPDC algorithmic framework.  All algorithms are implemented in MATLAB R2024b and executed on a Mac mini with an Apple M4 Pro processor and 24GB of memory. 

\subsection{Two DC-type SPCA models}
Let $A \in \Rbb^{m \times n}$ be the data matrix with $m$ samples and $n$ attributes.  Although PCA is popular in dimensionality reduction,  
its limited interpretability has motivated the  
development of sparse PCA (SPCA)  
\cite{d2004direct,d2008optimal,journee2010generalized,chen2020proximal,cai2021note}.  
 SPCA can be formulated either as the $\ell_0$-regularized model
\be \label{prob:spca:l0:reg}
\min_{X \in \stief}   -\mtr\!\left( X^\tran A^\tran A  X\right) + \sigma \|X\|_0,
\ee
or as the $\ell_0$-constrained model
\be \label{prob:spca:l0:constrained}
\min_{X \in \stief}   -\mtr\!\left( X^\tran A^\tran A  X\right)\quad \st \quad \|X\|_0 \leq k,
\ee
where $k \geq r$ is a prescribed sparsity parameter and $\sigma > 0$ is a regularization parameter. 

To address the computational challenge posed by the $\ell_0$ models, a widely adopted relaxation is the $\ell_1$-SPCA (see, e.g., \cite{chen2020proximal}):
\be \label{prob:spca:l1:reg}
\min_{X \in \stief}   -\mtr\!\left( X^\tran A^\tran A  X\right) + \gamma \|X\|_1,
\ee
where $\gamma > 0$ is a regularization parameter. This model serves as the baseline in our experiments.  
While computationally tractable, this relaxation may fail to faithfully capture the sparsity structures of the original $\ell_0$ models. Motivated by the equivalence results in Section \ref{section:DC:sparse}, we introduce two DC-type relaxations that are more closely connected to the $\ell_0$ formulations:  
(i) \emph{Capped-$\ell_1$-SPCA} 
\be\label{equ:spca:capped:l1} 
\min_{X \in \stief}   -\mtr\!\left( X^\tran A^\tran A  X\right) + \gamma \sum_{ij} \min\{\upsilon |X_{ij}|, 1\}, 
\ee
and (ii) \emph{$\ell_1$-$\ell_{[k]}$-SPCA} 
\be\label{equ:spca:laregest:knorm} 
\min_{X \in \stief}   -\mtr\!\left( X^\tran A^\tran A  X\right) + \gamma \big( \|X\|_1 - \normmm{X}_k\big),
\ee
where $\upsilon > 0$.  
Both models are instances of the general DC formulation \eqref{equ:prob:dc} and yield asymptotically exact relaxations of \eqref{prob:spca:l0:reg} and \eqref{prob:spca:l0:constrained}, respectively.  
Unlike the $\ell_1$-regularized model, these DC-type relaxations admit finite-parameter equivalence on the sphere. 
In particular, the parameter $\gamma$ in the capped-$\ell_1$-SPCA equals the sparsity penalty $\sigma$ in the $\ell_0$-regularized model \eqref{prob:spca:l0:reg}. 

\subsection{Numerical results on DC-type DCA}
In this subsection, we present numerical results  
 to illustrate the modeling effectiveness of  capped-$\ell_1$- and $\ell_1$-$\ell_{[k]}$-SPCA in comparison with the $\ell_1$-SPCA baseline, as well as the computational efficiency of the proposed iRPDC algorithms against OADMM  \cite{yuan2024admm}.

\subsubsection{Experimental setup} 
We consider two types of data matrices $A$:
(i) random instances generated as in \cite{zhou2023semismooth} with $m=500$ and $n=4000$;
(ii) the real dataset \texttt{cifar10} from LIBSVM \cite{CC01a}, with $m=500$ randomly selected rows and feature dimension $n=3072$.
For each type, results over 20 independent instances are reported.  Two evaluation metrics are used: the scaled variance ${\rm v_{sc}} = \|A\bar X\|_{\Ftt}^2/\|AX^{\rm pca}\|_{\Ftt}^2$, where $\bar X$ is the solution of a tested model and $X^{\rm pca}$ is the PCA solution of \eqref{prob:spca:l0:reg} with $\sigma=0$ (obtained via SVD of $A$); and the sparsity level ${\rm s_p}$, defined as the percentage of zero entries in $\bar X$. The penalty parameter $\gamma$ in \eqref{prob:spca:l1:reg}, \eqref{equ:spca:capped:l1}, and \eqref{equ:spca:laregest:knorm} is set as 
$\gamma = \widetilde \gamma \|A X^{\rm pca}\|_{\Ftt}^2/(nr)$, where $\widetilde \gamma > 0$ will be specified later.  

For the iRPDC algorithms, we set $\epsilon = 10^{-4}$, $\rho = 0.99$, $c = 10^{-4}$, $s = 0.5$, $\beta_1 = 0.99/(2+8c)$, $\omega_0 = 2 \times 10^{-5}  L_h^0$, $a = 1.5$, $L_{\min} = 10^{-10}L$, and $L_{\max} = 10^{10} L$.  For iRPDC-BB,  we use $\varrho_2 = 100$ and $\varrho_1 = 10^{-4}$.
The adaptive $\ell_j$ follows the Riemannian BB stepsize \cite{wen2013feasible} as
$\max\{L_{\min}, \min\{\iprod{p_j}{p_j}/|\iprod{p_j}{x_j - x_{j-1}}|, L_{\max}\}\}.$  We also enforce a lower bound on the subproblem tolerance by setting $\varepsilon_j \leftarrow \max\{\varepsilon_j,10^{-10}\}$. The iRPDC algorithms terminate when the prescribed stopping conditions are met, or when both $\|x_j - x_{j-1}\|_{\Ftt} \leq 10^{-4}\sqrt{r}$ and $|F(x_j) - F(x_{j-1})| \leq 10^{-6}\max\{1, |F(x_j)|\}$ hold, or  $100$ iterations are reached (warm starts are used across problem sequences; see the subsequent experiments).

\subsubsection{Modeling effectiveness}  \label{subsection:numerical:approximation}  We  evaluate the approximation quality of  the  capped-$\ell_1$-SPCA \eqref{equ:spca:capped:l1} and the $\ell_1$-$\ell_{[k]}$-SPCA \eqref{equ:spca:laregest:knorm}, in comparison with  the widely used $\ell_1$-SPCA baseline \eqref{prob:spca:l1:reg}. All problems are solved using iRPDC-BB.

\paragraph{Capped-$\ell_1$-SPCA}   
This model can be viewed as a refinement of $\ell_1$-SPCA \eqref{prob:spca:l1:reg}, since setting $\upsilon = 1$ reduces the capped-$\ell_1$ term to $\|X\|_1$ for $X \in \stief$.  For each fixed $\widetilde \gamma$, we solve a sequence of problems with $\upsilon \in \{1, 1.5, 1.5^2, \ldots, 1.5^{20}\}$, starting from $\upsilon = 1$ initialized at $X^{\rm PCA}$ and warm-starting subsequent problems.  For each dataset, two values of $\widetilde \gamma$ are considered: one such that capped-$\ell_1$-SPCA with large $\upsilon$ achieves sparsity around $0.8$, and the other such that $\ell_1$-SPCA yields a solution of comparable sparsity. Fig.~\ref{figure:cappedl1} shows that capped-$\ell_1$-SPCA better approximates the $\ell_0$-regularized model than $\ell_1$-SPCA. Specifically, the objective value of \eqref{prob:spca:l0:reg} decreases monotonically as $\upsilon$ increases, and eventually stabilizes along with sparsity and variance.  Moreover, to reach the same sparsity, $\ell_1$-SPCA requires a much larger $\widetilde \gamma$ but attains a lower variance.  
For example, on the random dataset,  capped-$\ell_1$-SPCA with $\widetilde \gamma = 1.2$ (with sufficiently large $\upsilon$) and $\ell_1$-SPCA with $\widetilde \gamma = 120$ (with $\upsilon = 1$) both yield sparsity about $0.8$, but the variances are  $0.7563$ versus $0.6167$, respectively.  
On  \texttt{cifar10}, the corresponding values are $0.9622$ versus $0.9416$. 

\paragraph{$\ell_1$-$\ell_{[k]}$-SPCA}
This model serves as a more direct relaxation of the $\ell_0$-constrained SPCA~\eqref{prob:spca:l0:constrained}.  We solve a series of problems with $\widetilde \gamma \in \{1, 1.5, 1.5^2, \ldots, 1.5^{20}\}$,  where the case $\widetilde \gamma = 1$ is initialized at $X^{\rm PCA}$  and each subsequent one is warm-started. The comparison results are plotted in Fig.~\ref{figure:largestknorm}, where the label $\ell_1$-$\ell_{[k]}$-${\rm s}_{\mathrm{p}}$ indicates that $k = (1 - {\rm s}_{\mathrm{p}})nr$ in~\eqref{equ:spca:laregest:knorm}. The figure shows that $\ell_1$-$\ell_{[k]}$-SPCA with sufficiently large $\widetilde \gamma$ consistently attains solutions whose sparsity levels match the $\ell_0$-norm constraint in~\eqref{prob:spca:l0:constrained}.
In contrast, $\ell_1$-SPCA with large $\widetilde \gamma$ often degenerates, typically returning $r$ columns of the $n \times n$ identity matrix.
For instance, on the random dataset, when $\widetilde \gamma \geq 438$, model \eqref{equ:spca:laregest:knorm} stabilizes at sparsity ${\rm s}_{\mathrm{p}} = 0.7$ with variance $0.8099$, whereas $\ell_1$-SPCA yields  ${\rm s}_{\mathrm{p}} = 0.9997$ but variance only $0.0063$.
Even with careful tuning (e.g., $\widetilde \gamma \approx 86$ for $\ell_1$-SPCA), the resulting solution ${\rm s}{\mathrm{p}} \approx 0.7$ with variance $0.7114$ remains inferior to that of $\ell_1$-$\ell_{[k]}$-SPCA.  Similar trends are observed across other sparsity levels and datasets. 

In summary, capped-$\ell_1$-SPCA \eqref{prob:spca:l1:reg}  offers a tighter approximation to the $\ell_0$-regularized model \eqref{prob:spca:l0:reg}, while $\ell_1$-$\ell_{[k]}$-SPCA \eqref{equ:spca:laregest:knorm} more effectively captures the constraints of the $\ell_0$-constrained model \eqref{prob:spca:l0:constrained}. Together, these results indicate that both DC formulations reflect the structure of their respective $\ell_0$ counterparts more faithfully than the standard $\ell_1$ relaxation.

\begin{figure}[t]
\centering
\subfloat[Objective value]{%
  \begin{minipage}{0.32\linewidth}
    \centering
    \includegraphics[width=\linewidth,height=0.5\linewidth]{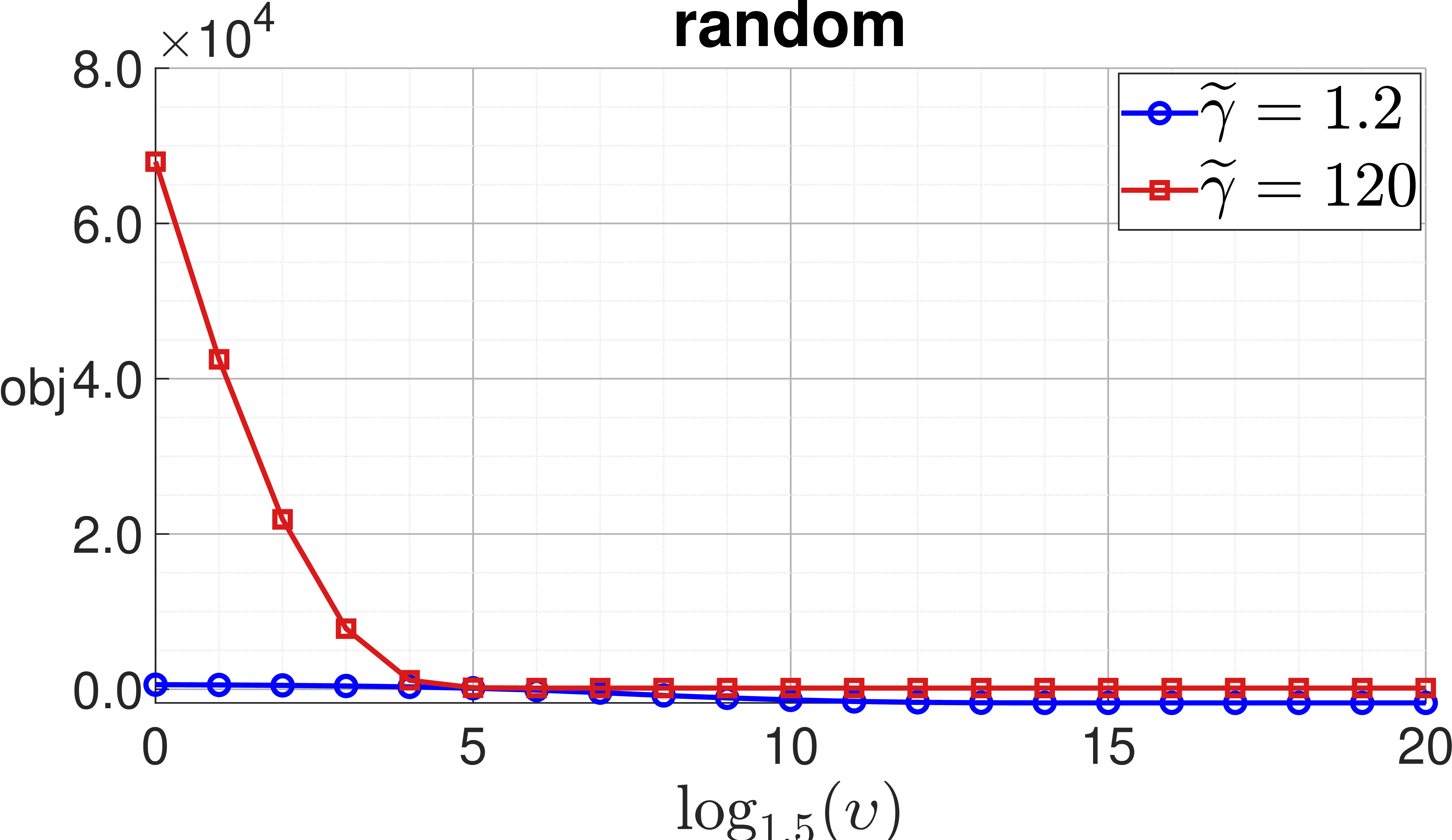}\\[1pt]
    \includegraphics[width=\linewidth,height=0.5\linewidth]{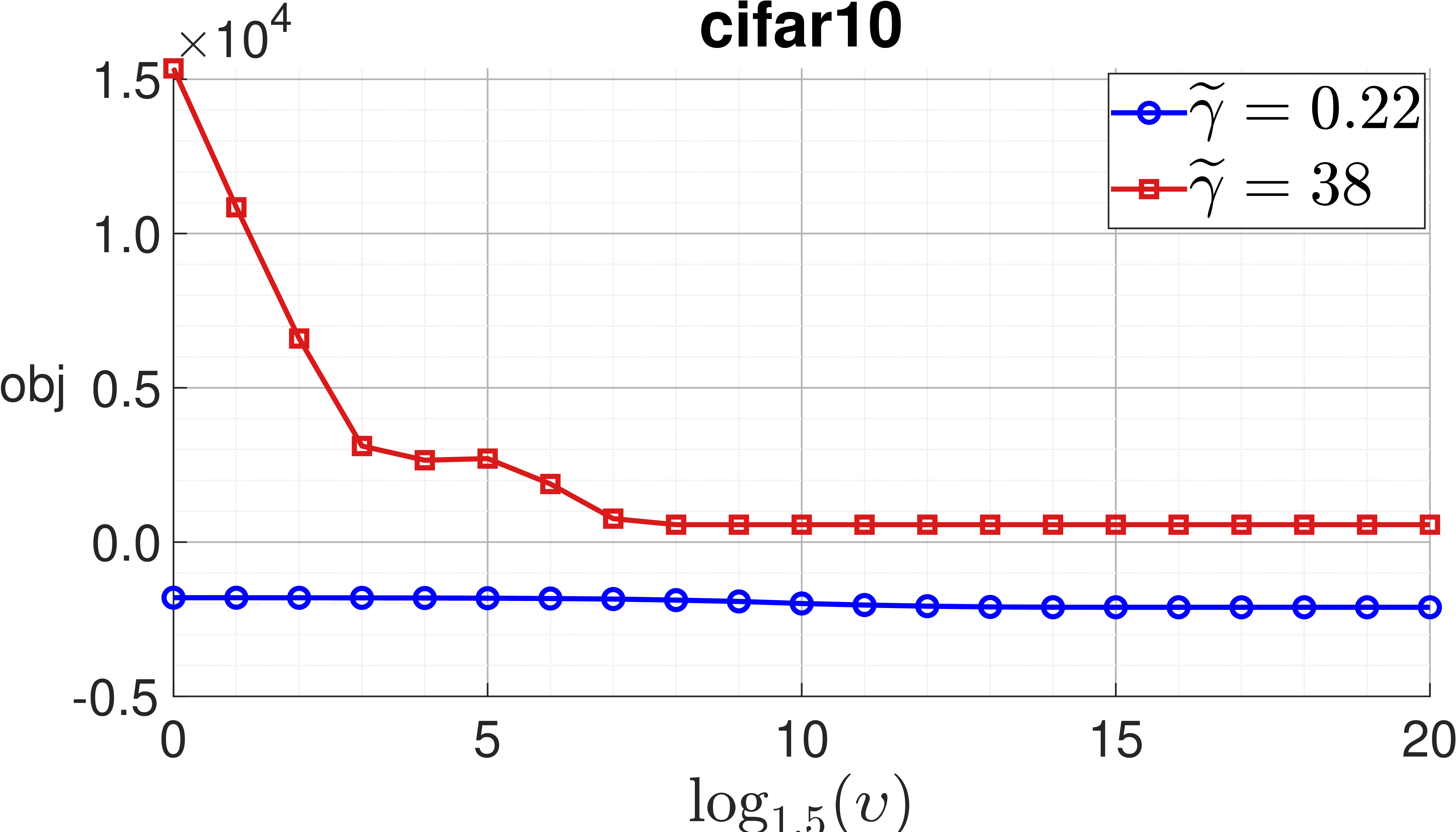}
  \end{minipage}
}
\hfill
\subfloat[Sparsity level ${\rm s}_p$]{%
  \begin{minipage}{0.32\linewidth}
    \centering
    \includegraphics[width=\linewidth,height=0.5\linewidth]{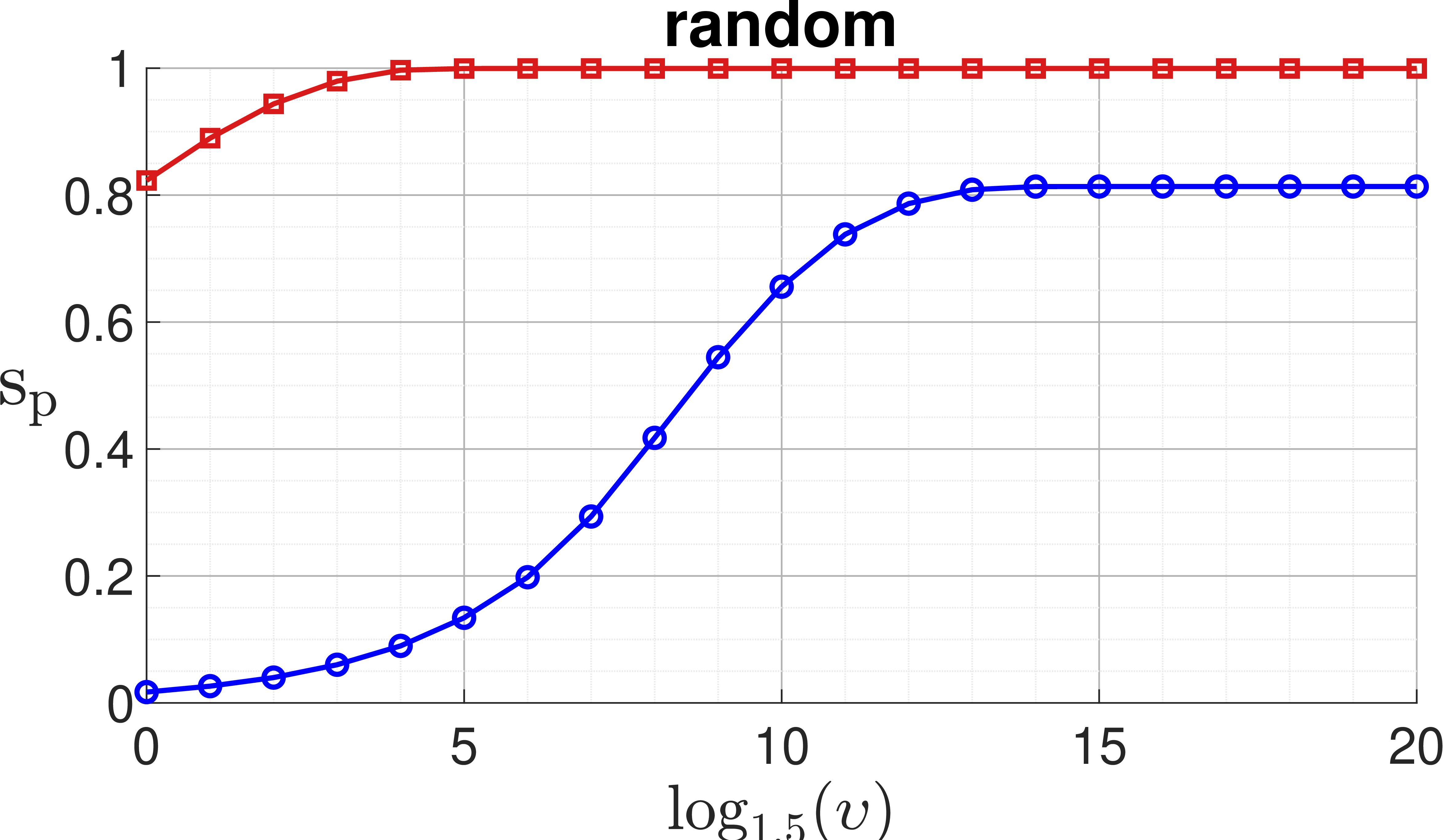}\\[1pt]
    \includegraphics[width=\linewidth,height=0.5\linewidth]{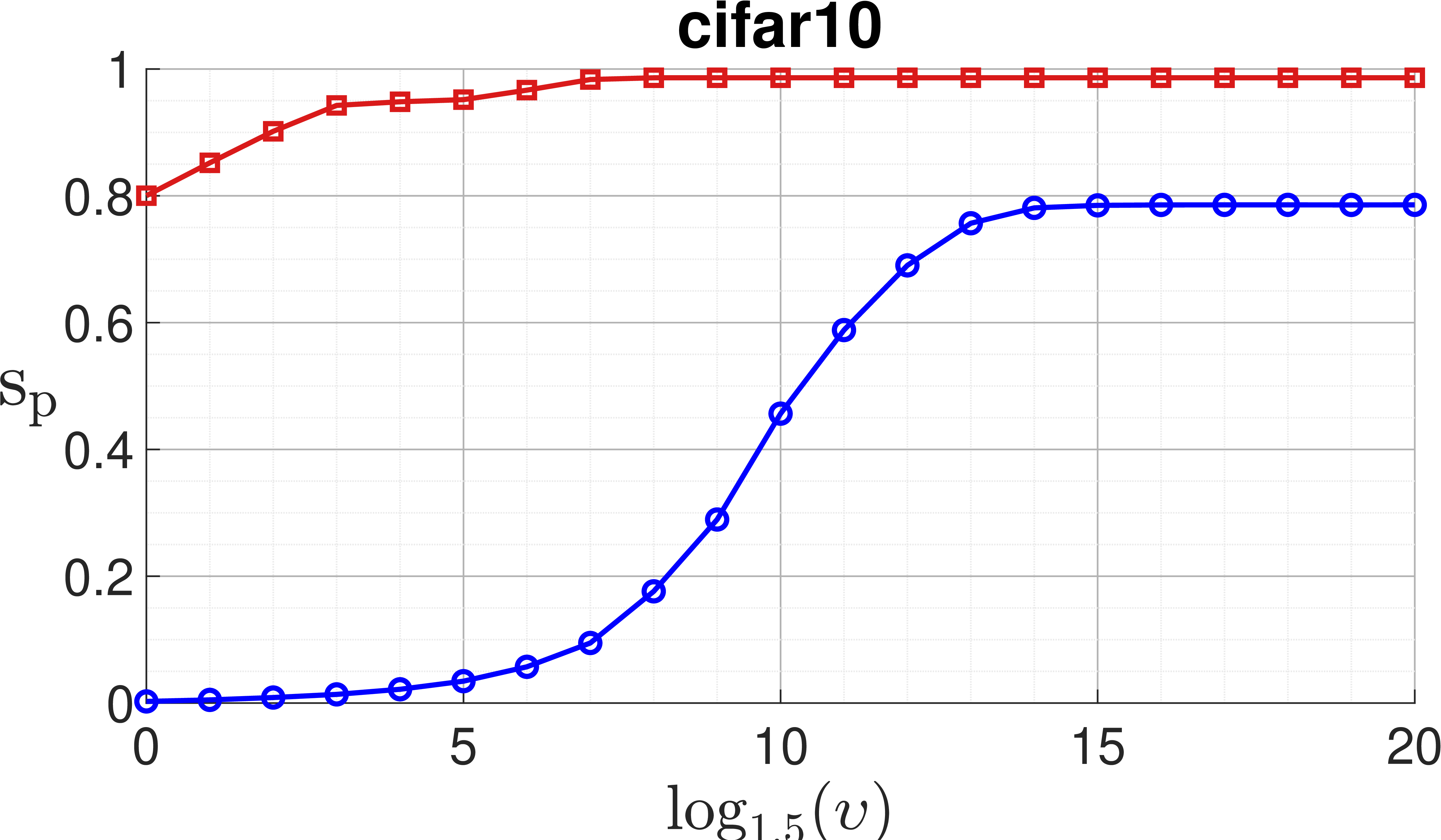}
  \end{minipage}
}
\hfill
\subfloat[Scaled variance ${\rm v_{sc}}$]{%
  \begin{minipage}{0.32\linewidth}
    \centering
    \includegraphics[width=\linewidth,height=0.5\linewidth]{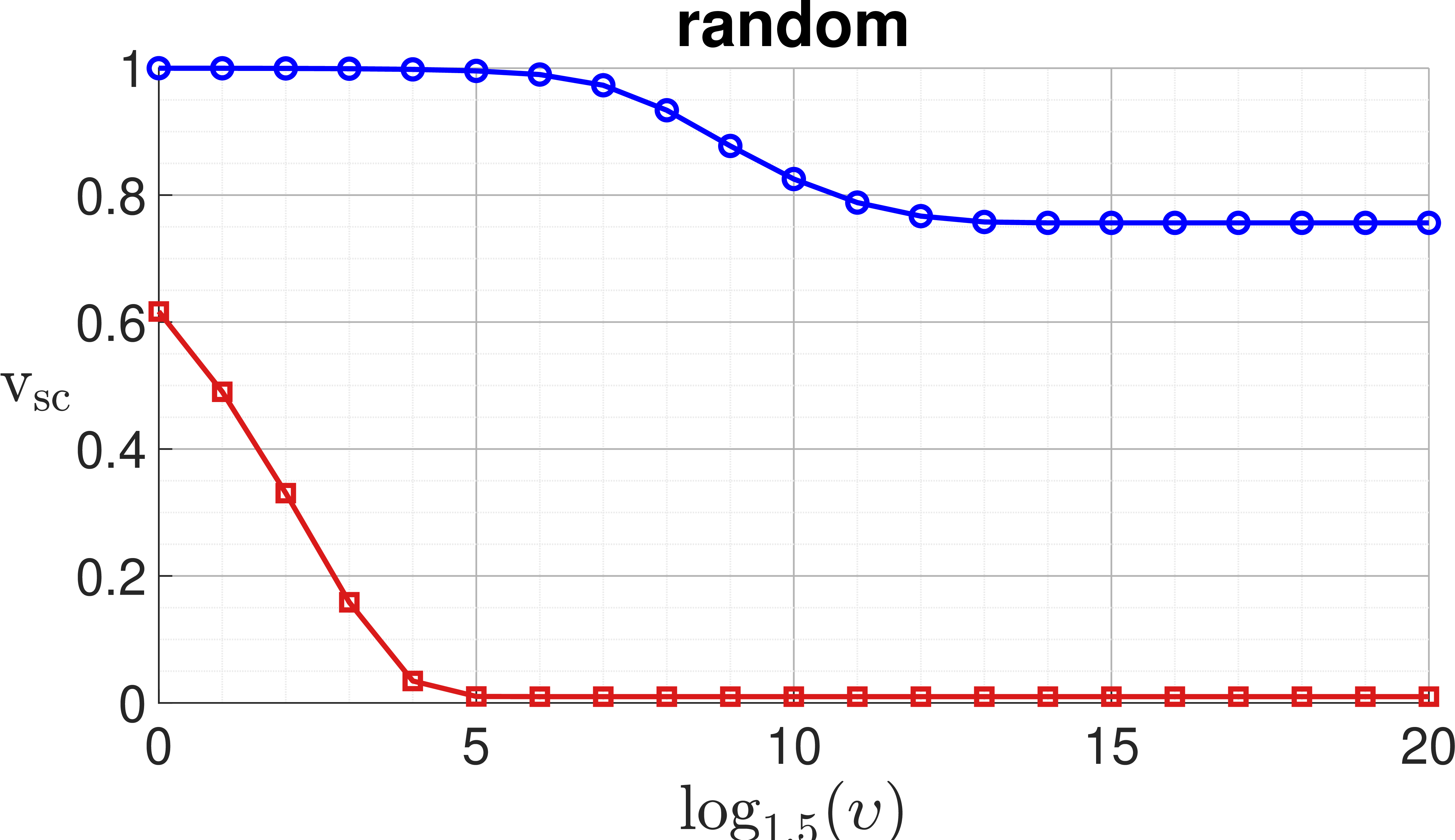}\\[1pt]
    \includegraphics[width=\linewidth,height=0.5\linewidth]{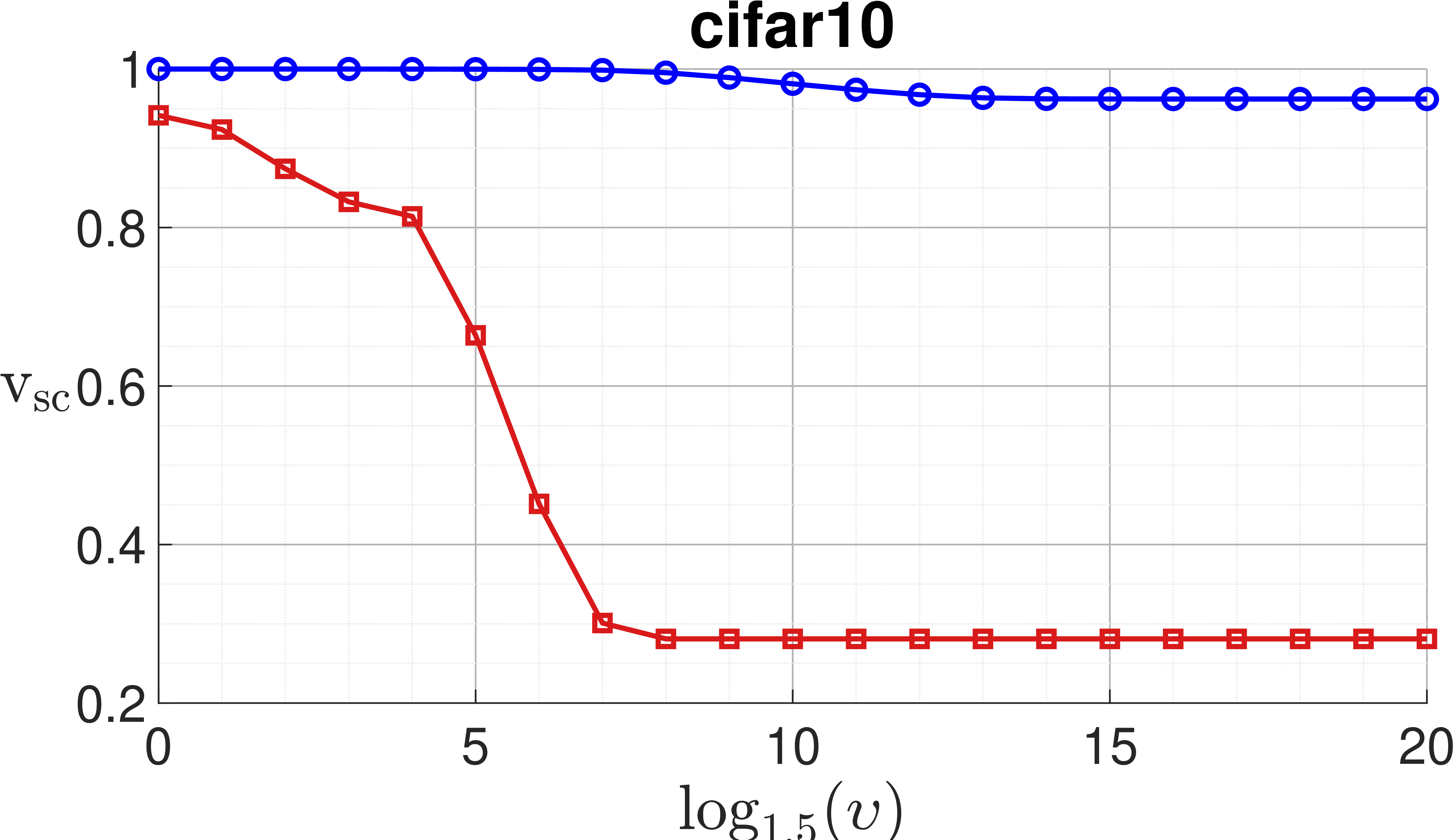}
  \end{minipage}
}
\caption{Results for capped-$\ell_1$-SPCA~\eqref{equ:spca:capped:l1} with $r=20$.}
\label{figure:cappedl1}
\end{figure}

\begin{figure}[t]
\centering
\subfloat[Sparsity level ${\rm s}_p$]{
  \begin{minipage}{0.32\linewidth}
    \centering
    \includegraphics[width=\linewidth,height=0.5\linewidth]{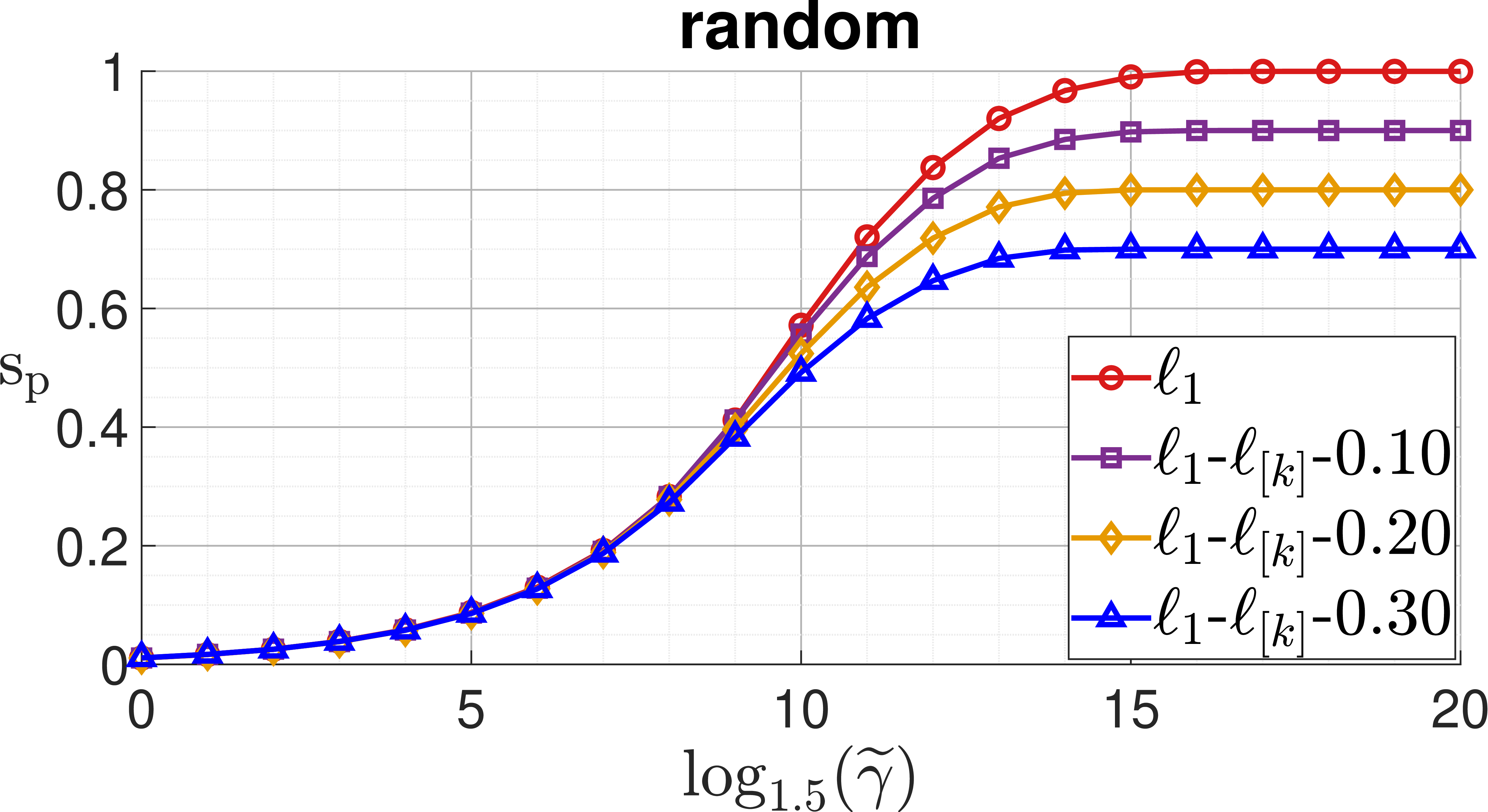}\\[1pt]
    \includegraphics[width=\linewidth,height=0.5\linewidth]{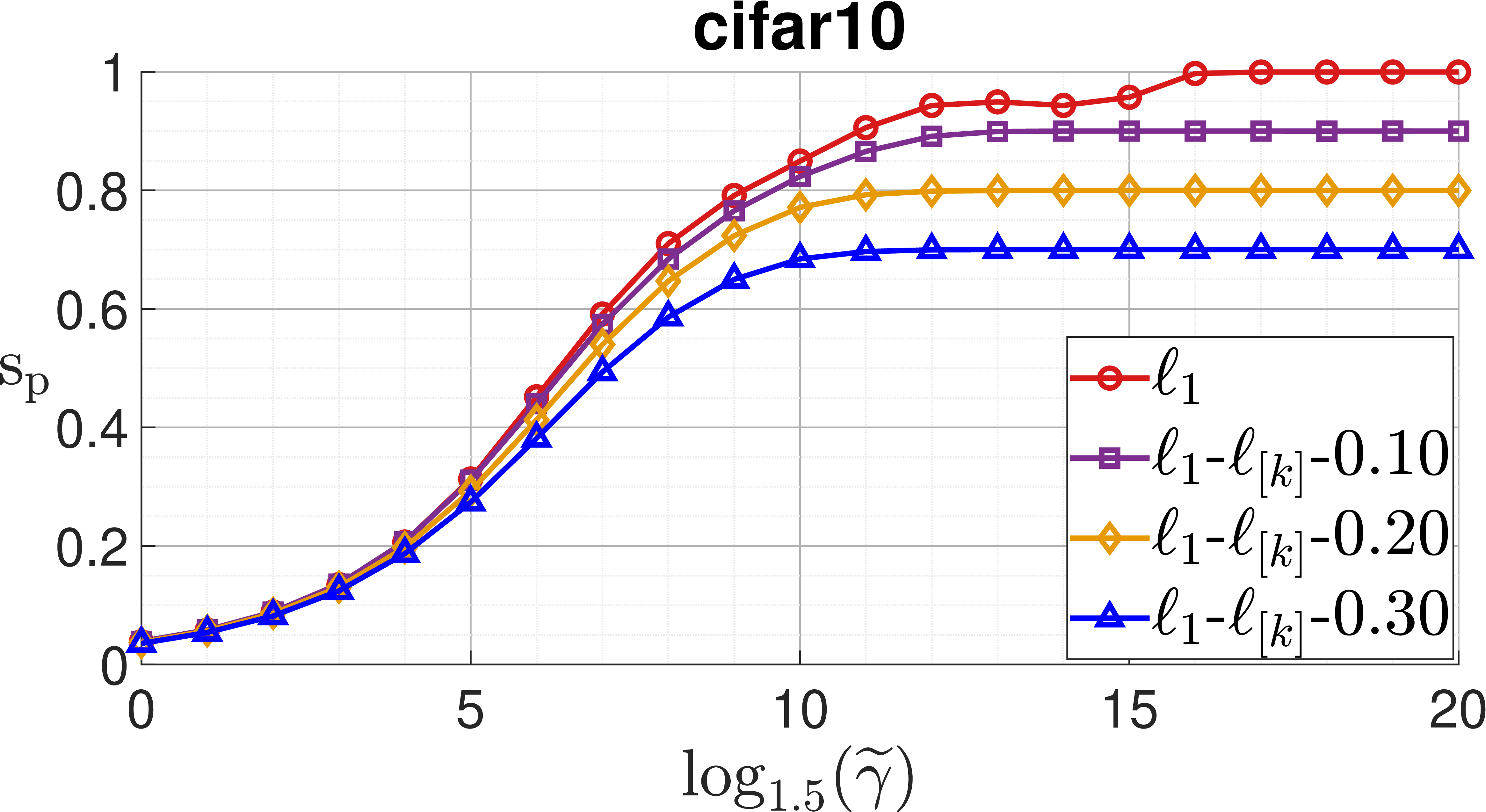}
  \end{minipage}
}
~~~
\subfloat[Scaled variance ${\rm v_{sc}}$]{ 
  \begin{minipage}{0.32\linewidth}
    \centering
    \includegraphics[width=\linewidth,height=0.5\linewidth]{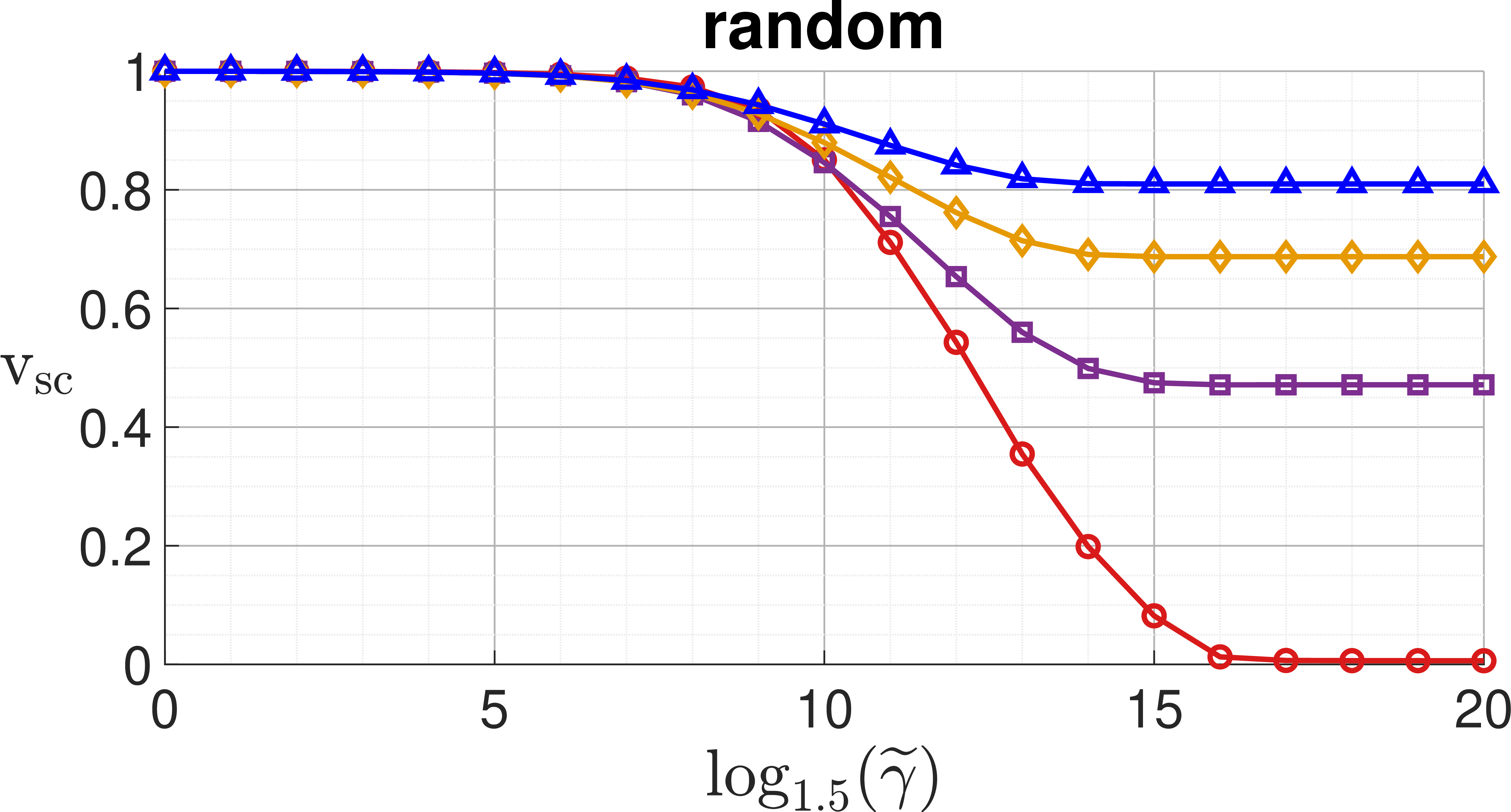}\\[1pt]
    \includegraphics[width=\linewidth,height=0.5\linewidth]{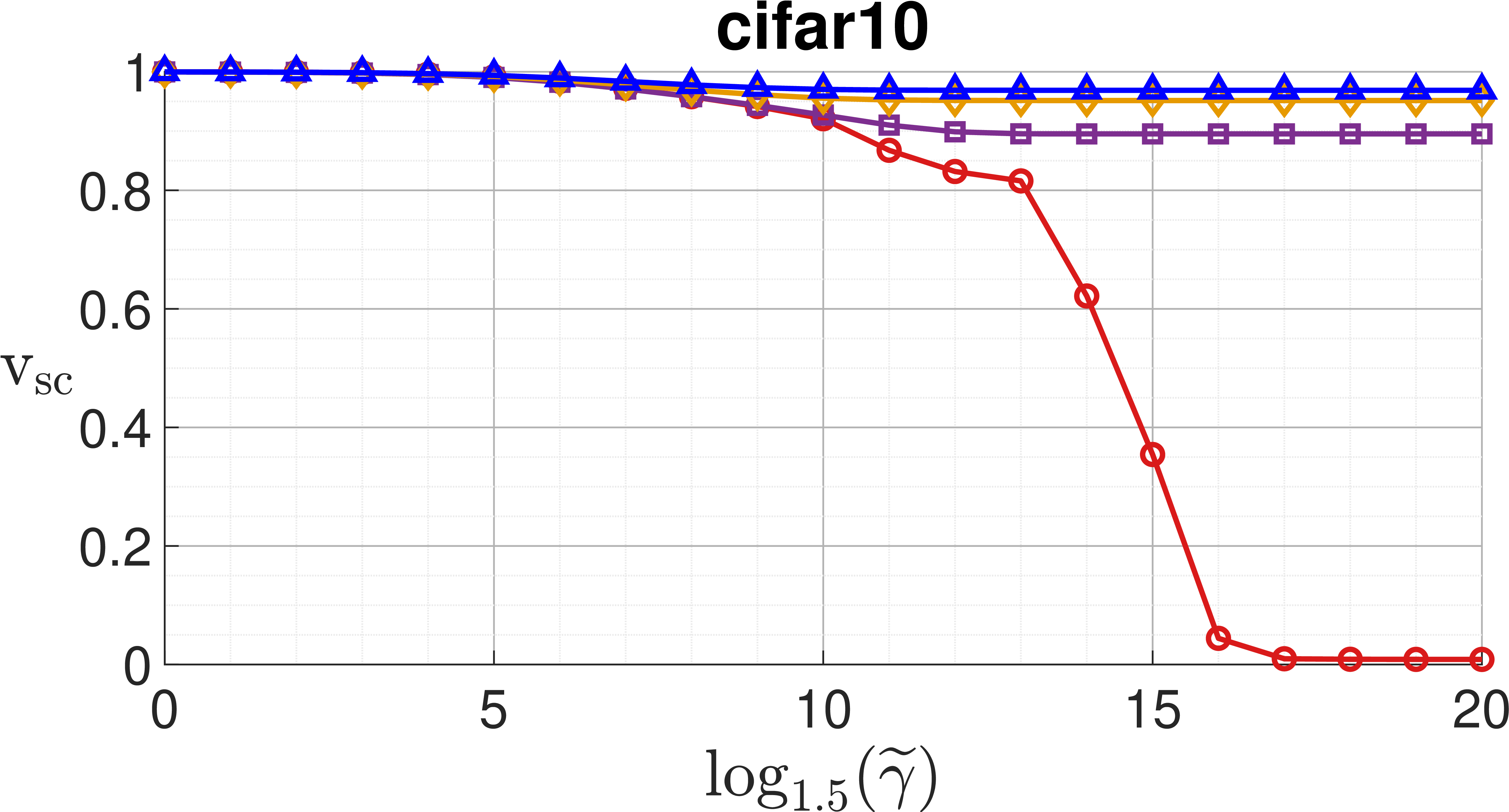}
  \end{minipage}
}
\caption{Results for $\ell_1$-$\ell_{[k]}$-SPCA~\eqref{equ:spca:laregest:knorm} with $r=20$.}
\label{figure:largestknorm}
\end{figure}
\subsubsection{Computational efficiency} \label{subsection:numerical:efficiency}
Among the three iRPDC algorithms introduced in Section \ref{subsection:iRPDCA},  we focus on iRPDC-NFG and iRPDC-BB, as they adaptively adjust the number of inner iterations instead of fixing them in advance. We also include a semismooth Newton-based algorithm, denoted by iRPDC-ASSN, where the dual subproblem is solved by a semismooth Newton method~\cite{xiao2018regularized,chen2020proximal}, implemented following \cite{huang2024riemannian} (code available at \url{https://www.math.fsu.edu/whuang2}). As an external baseline, we compare with OADMM \cite{yuan2024admm} (code available at \url{https://openreview.net/forum?id=K1G8UKcEBO}).

For capped-$\ell_1$-SPCA~\eqref{equ:spca:capped:l1}, we solve a sequence of problems with parameters $\upsilon$ starting from $1$ and increasing geometrically by a factor of $1.5$ (i.e., $1,1.5,1.5^2,\ldots$), terminating once the relative change in objective values falls below $10^{-4}$ and the change in sparsity below $10^{-3}$.  For $\ell_1$-$\ell_{[k]}$-SPCA~\eqref{equ:spca:laregest:knorm}, we vary $\widetilde\gamma$ in the same way until the solution achieves sparsity within $10^{-3}$ of the target level. For each fixed parameter, OADMM terminates when $\|x_j - x_{j-1}\|_{\Ftt} \leq 10^{-4}\sqrt{r}$ and $|F(x_j) - F(x_{j-1})| \leq 10^{-6}\max\{1,|F(x_j)|\}$, or after $500$ iterations. For OADMM, we use the recommended parameter settings from \cite{yuan2024admm}, except that the key penalty parameter $\beta_0$ is carefully tuned. In the initial problems ($\upsilon=1$ or $\widetilde \gamma=1$), $\beta_0$ is chosen from the candidate set $10^{\tilde \beta_0}$ with $\tilde \beta_0 \in \{0,0.2,\ldots,4\}$ to maximize OADMM’s objective performance. For subsequent problems, $\beta_0$ is initialized from the previous solution and scaled by $1.5$, instead of being re-tuned each time. In practice, for capped-$\ell_1$-SPCA, we set $\tilde \beta_0 = 1.4, 1.2, 1.0, 1.0$ for $r = 20, 40, 80, 100$ on the random dataset, and $\tilde \beta_0 = 2.6, 2.2, 1.6, 1.6$ on the \texttt{cifar10} dataset. For $\ell_1$-$\ell_{[k]}$-SPCA, we fix $\tilde \beta_0$ at $2.6$ for the random dataset and $3.2$ for the \texttt{cifar10} dataset. 
\begin{table}[!t]
\footnotesize
\caption{Comparison of OADMM and iRPDC algorithms on capped-$\ell_1$-SPCA~\eqref{equ:spca:capped:l1} 
and $\ell_1$-$\ell_{[k]}$-SPCA~\eqref{equ:spca:laregest:knorm} with $k = 0.2 nr$. 
Methods ``a'', ``b'', ``c'', ``d'' denote OADMM, iRPDC-ASSN, iRPDC-NFG, and iRPDC-BB, respectively. }
\label{table:combined}
\centering
\setlength{\tabcolsep}{2pt}
\resizebox{\textwidth}{!}{
\begin{tabular}{@{} c  cccc  cccc   cccc @{}}
\toprule
&  \multicolumn{4}{c}{${\rm iter}_{\rm out}$ ($\overline{\rm iter}_{\rm in}$)} 
& \multicolumn{4}{c}{time (${\rm time}_{\rm sub}$)} 
& \multicolumn{4}{c}{obj}  
\\
\cmidrule(lr){2-5}\cmidrule(lr){6-9}\cmidrule(lr){10-13}  
\multicolumn{1}{c}{\(r\)}   
& a & b & c  & d
& a & b & c  & d
& a & b & c  & d
\\
\midrule
\addlinespace[0.4ex]
\multicolumn{13}{c}{{capped-$\ell_1$-SPCA}:  random, $n=4000$, $\widetilde \gamma = 0.6$}  \\[2pt]
20 & 5315 & 983(1.2) & 974(2.3) & 976(1.1)  & 24 & 8(3) & 7(2) & 6(2) & -2238 & -2307 & -2307 & -2307 \\ 
40 & 5935 & 1103(1.2) & 1114(6.1) & 1097(1.9)  & 41 & 23(15) & 22(14) & 15(8) & -2650 & -2754 & -2754 & -2754 \\ 
80 & 7328 & 1291(1.2) & 1310(31.3) & 1294(5.7)  & 125 & 155(119) & 235(199) & 93(57) & -2998 & -3123 & -3122 & -3122 \\ 
100 & 7082 & 1379(1.3) & 1405(49.9) & 1370(8.1)  & 122 & 182(143) & 364(323) & 123(83) & -3099 & -3223 & -3223 & -3223 \\
\midrule[0.2pt]
\addlinespace[0.4ex]
\multicolumn{13}{c}{{capped-$\ell_1$-SPCA}: cifar10, $n=3072$, $\widetilde \gamma = 0.1$} \\[2pt]
20 & 3717 & 1035(0.6) & 1057(1.9) & 1038(0.8)  & 23 & 9(2) & 9(2) & 8(1) & -2118 & -2183 & -2183 & -2183 \\ 
40 & 4525 & 949(1.0) & 935(2.4) & 917(1.0)  & 39 & 24(14) & 18(8) & 16(6) & -2377 & -2430 & -2433 & -2433 \\ 
80 & 4914 & 831(0.7) & 838(2.5) & 830(1.0)  & 62 & 39(23) & 27(11) & 26(9) & -2592 & -2633 & -2633 & -2633 \\ 
100 & 4893 & 788(0.7) & 791(2.6) & 786(1.1)  & 74 & 44(26) & 31(13) & 28(11) & -2654 & -2689 & -2689 & -2689 \\ 
\midrule[0.2pt]
\addlinespace[0.4ex]
 \multicolumn{13}{c}{$\ell_1$-$\ell_{[k]}$-SPCA: random, $n=4000$}\\[2pt]
20 & 9261 & 1283(1.3) & 1286(5.8) & 1271(1.7)  & 67 & 14(5) & 15(5) & 12(3)& -2117 & -2373 & -2372 & -2373 \\ 
40 & 11011 & 1457(1.4) & 1459(9.1) & 1457(2.4)  & 122 & 47(28) & 51(32) & 35(15)& -2602 & -2987 & -2987 & -2987 \\ 
80 & 11802 & 1399(1.3) & 1400(11.1) & 1402(2.8)  & 217 & 111(79) & 113(81) & 62(30)& -3090 & -3473 & -3473 & -3473 \\ 
100 & 11316 & 1417(1.3) & 1419(11.8) & 1413(3.0)  & 234 & 126(87) & 137(98) & 73(34)& -3226 & -3591 & -3591 & -3591 \\ 
\midrule[0.2pt]
\addlinespace[0.4ex]
 \multicolumn{13}{c}{$\ell_1$-$\ell_{[k]}$-SPCA: cifar10, $n=3072$}\\[2pt] 
20 & 8148 & 1121(1.0) & 1119(10.6) & 1115(3.3)  & 41 & 9(4) & 12(6) & 9(3)& -1817 & -2192 & -2192 & -2192 \\ 
40 & 8588 & 1214(0.9) & 1221(18.4) & 1212(3.9)  & 71 & 29(17) & 48(36) & 25(13)& -2098 & -2455 & -2455 & -2455 \\ 
80 & 8660 & 1296(1.0) & 1295(16.6) & 1295(3.4)  & 115 & 64(43) & 102(81) & 44(23)& -2369 & -2627 & -2627 & -2627 \\ 
100 & 8402 & 1312(1.0) & 1312(16.9) & 1309(3.4)  & 138 & 74(47) & 131(104) & 52(26)& -2447 & -2676 & -2676 & -2677 \\
\bottomrule 
\end{tabular}
}
\end{table}

Table~\ref{table:combined} reports the comparison results. Here,  ``$\text{iter}_{\text{out}}$'' denotes the number of outer iterations, ``$\overline{\text{iter}}_{\text{in}}$'' represents the average inner iterations per outer iteration, ``time'' refers to the total runtime in seconds (measured by \texttt{tic}–\texttt{toc}), and ``${\rm time}_{\rm sub}$'' means the time spent solving subproblems. 
The column ``obj'' reports the objective value of \eqref{prob:spca:l0:reg} for capped-$\ell_1$-SPCA, and that of \eqref{prob:spca:l0:constrained} for $\ell_1$-$\ell_{[k]}$-SPCA. In the latter case, the objective coincides with the negative variance.

Several observations can be made. First, among the three iRPDC algorithms, iRPDC-ASSN requires the fewest inner iterations per outer iteration, followed by iRPDC-BB and then iRPDC-NFG. However, due to the high cost of each semismooth Newton step, iRPDC-ASSN is not the most efficient in runtime. Instead, iRPDC-BB achieves the best efficiency, with its advantage becoming more evident as the subspace dimension $r$ increases.  Second, compared with the baseline OADMM, the iRPDC algorithms consistently deliver much better solution quality, achieving lower objective values across all tested settings. In terms of efficiency, iRPDC-BB shows a clear advantage, converging with substantially fewer outer iterations and shorter runtimes. For instance, on the \texttt{cifar10} dataset, when solving $\ell_1$-$\ell_{[k]}$-SPCA with $r=100$, iRPDC-BB takes  52 seconds compared to 138 seconds for OADMM, while also yielding a significantly better variance (2677 vs.~2447). iRPDC-ASSN exhibits runtime performance comparable to OADMM, while iRPDC-NFG is slower; nevertheless, both still yield higher-quality solutions. Finally, it is worth noting that OADMM's performance is sensitive to the choice of $\beta_0$,
whereas the iRPDC algorithms maintain stable performance without requiring such delicate parameter tuning.

 \section{Conclusions}  \label{section:concluding}
In this paper, we studied a new class of nonsmooth Riemannian DC optimization problems. 
We established equivalence results between the Riemannian DC formulations and their sparse counterparts on specific manifolds, which motivates the development of efficient algorithms for such problems. 
We then proposed the iRPDC algorithmic framework with convergence guarantees. 
Within this framework, we developed practical iRPDC algorithms that inexactly solve the regularized dual subproblems using either the NFG method, the BB method, or the AR scheme. 
A key feature of our proposed framework is that the subproblem tolerance is determined adaptively from the information of previous iterates, which not only ensures flexibility in solving the subproblems but also enables a linesearch procedure that adaptively captures the local curvature. 
This mechanism, to the best of our knowledge, has not been explicitly considered in existing inexact Riemannian proximal algorithms.  We showed that the iRPDC algorithms attain an $\epsilon$-Riemannian critical point within $\mathcal{O}(\epsilon^{-2})$ outer iterations, with overall iteration complexities of  $\mathcal{O}(\epsilon^{-4})$, $\mathcal{O}(\epsilon^{-3}\log \epsilon^{-1})$, and $\mathcal{O}(\epsilon^{-3})$ for the three specific algorithms iRPDC-BB, iRPDC-NFG,  and iRPDC-AR, respectively.
Even in the special case when $g(\cdot)=0$, the iRPDC algorithm reduces to a new Riemannian proximal-type algorithm with such theoretical guarantees.
Numerical results on SPCA with DC terms validate both the effectiveness of the Riemannian DC models and the efficiency of the proposed algorithms. 
\bibliographystyle{siamplain}
\bibliography{dc-manifold}
\end{document}